\documentclass[12pt,reqno]{amsart}
\usepackage[margin=1.2in]{geometry}
\usepackage{graphicx,latexsym}
\usepackage{comment}
\usepackage{amsfonts, amssymb, amsmath, amsthm, bm, hyperref}
\usepackage{relsize}
\usepackage{esvect}

\newcommand{\NN}{\mathbb N}

\newcommand{\CC}{\mathbb C}
\newcommand{\RR}{\mathbb R}
\newcommand{\ZZ}{\mathbb Z}

\newcommand{\SSS}{\mathcal S}

\newcommand{\ad}{\operatorname{ad}}

\newtheorem{theorem}{Theorem}[section]

\newtheorem{lemma}[theorem]{Lemma}

\theoremstyle{remark}
\newtheorem{remark}[theorem]{Remark}

\theoremstyle{definition}
\newtheorem{definition}[theorem]{Definition}

\allowdisplaybreaks
\numberwithin{equation}{section}

\newcommand{\beq}{\begin{eqnarray}}
\newcommand{\eeq}{\end{eqnarray}}

\newcommand{\beqs}{\begin{eqnarray*}}
\newcommand{\eeqs}{\end{eqnarray*}}

\begin{document}

\author[S. Pilipovi\' c]{Stevan Pilipovi\' c}
\address{Department of Mathematics and Informatics,
University of Novi Sad, Trg Dositeja Obradovi\'{c}a 4, 21000 Novi Sad, Serbia}
\email{stevan.pilipovic@dmi.uns.ac.rs}

\author[B. Prangoski]{Bojan Prangoski}
\thanks{S. Pilipovi\' c is supported by the project 174024 of the MPNTR of Serbia  while the work of B. Prangoski was partially supported by the bilateral project ``Microlocal analysis and applications'' between the Macedonian and Serbian academies of sciences and arts.}
\address{Department of Mathematics, Faculty of Mechanical
Engineering-Skopje, University ``Ss. Cyril and Methodius", Karposh 2 b.b., 1000 Skopje, Macedonia}
\email{bprangoski@yahoo.com}

\title[Ellipticity and Fredholmness in the Weyl-H\"ormander calculus]{Equivalence of Ellipticity and Fredholmness in the Weyl-H\"ormander calculus}



\frenchspacing
\begin{abstract}
The main result is that the Fredholm property of a $\Psi$DO acting on Sobolev spaces in the Weyl-H\"ormander calculus and the ellipticity are equivalent for geodesically temperate H\"ormanders metrics whose associated Planck's functions vanish at infinity. Additionally, we prove that when the H\"ormander metric is geodesically temperate, and consequently the calculus is spectrally invariant, the inverse $\lambda\mapsto b_\lambda\in S(1,g)$ of every $\mathcal{C}^N$, $0\leq N\leq \infty$, map $\lambda\mapsto a_\lambda\in S(1,g)$ comprised of invertible elements on $L^2$ is again of class $\mathcal{C}^N$.
\end{abstract}
\maketitle

\section{Introduction}

The question of spectral invariance is of a significant importance in the theory of pseudodifferential operators. Recall that a pseudodifferential calculus is said to be spectrally invariant if for every $\Psi$DO with $0$ order symbol (consequently, continuous on $L^2$) which is invertible on $L^2$ its inverse is again a $\Psi$DO with a $0$ order symbol. This property has been proved by several authors for various global (and local) calculi including the Shubin calculus, the SG (scattering) calculus, the Beals-Fefferman calculus, e.t.c. (see \cite{beals3,CoifmanMeyer,cordes1,leo-sch,leo-sch1,Ueberberg}). In their seminal paper \cite{bon-che}, Bony and Chemin (see also \cite{bon-ler}) generalised these results by proving the spectral invariance for the Weyl-H\"ormander calculus \cite{hormander,horm3} when the H\"ormander metric satisfies the so-called geodesic temperance (see \cite{bon-che,lernerB}). In the first part of this article (Section \ref{vjtknc159}) we slightly improve two lemmas of \cite{bon-che,lernerB}, by, essentially, repeating the arguments employed there, but changing the right hand sides of the estimates in these results. Subsequently, we avail ourselves of these results to prove the following fact which sheds more light on the spectral invariance of the Weyl-H\"ormander calculus: the process of taking inverses in $S(1,g)$ preserves the regularity. To be more precise, if $\lambda\mapsto a_{\lambda}$ is a $\mathcal{C}^N$, $0\leq N\leq\infty$, mapping with values in $S(1,g)$ such that $a_{\lambda}^w$ is invertible on $L^2$, then the mapping $\lambda\mapsto b_{\lambda}$, where $b_{\lambda}^w$ is the inverse of $a_{\lambda}^w$, is also of class $\mathcal{C}^N$; in fact, we prove this result for matrix valued symbols. In the second part of the article (Section \ref{vnklst135}), we investigate the Fredholm properties of $\Psi$DOs with symbols in the Weyl-H\"ormander classes when acting between the Sobolev spaces naturally associated to them. The main result is that the Fredholm property of a $\Psi$DO can be characterised by the ellipticity of the symbol, that is a $\Psi$DO is Fredholm operator between appropriate Sobolev spaces if and only if its symbol is elliptic (see \cite{bog-sch,schrohe,leo-sch1,leo-sch} for similar type of results concerning special instances of the Weyl-H\"ormander calculus). This result heavily relies on the vanishing at infinity of the Planck function associated to the H\"ormander metric as well as on the main result of Section \ref{vjtknc159} which, in turn, depends on the spectral invariance and the geodesic temperance of the metric.

\section{Preliminaries}

Let $V$ be an $n$ dimensional real vector space with $V'$ being its dual. The $2n$-dimensional vector space is $W=V\times V'$ is symplectic with the symplectic form $[(x,\xi),(y,\eta)]=\langle \xi,y\rangle-\langle \eta,x\rangle$. We will always denote the points in $W$ with capital letters $X,Y,Z,\ldots$. Let $X\mapsto g_X$ be a Borel measurable symmetric covariant $2$-tensor field on $W$ that is positive definite at every point. We will always denote the corresponding positive definite quadratic form at $X\in W$ by the same symbol $g_X$, i.e. $g_X(T)=g_X(T,T)$, $T\in T_XW$. Denoting by $Q_X$ the corresponding linear map $W\rightarrow W'$ and by $\sigma:W\rightarrow W'$ the linear map induced by the symplectic form, one defines the symplectic dual of $Q_X$ by $Q^{\sigma}_X=\sigma^*Q_X^{-1}\sigma$. The corresponding symmetric covariant $2$-tensor field $X\mapsto g^{\sigma}_X$ is again Borel measurable and positive definite at every point; it can be given by $g^{\sigma}_X(T)=\sup_{S\in W\backslash\{0\}} [T,S]^2/g_X(S)$. We say that $X\mapsto g_X$ is a H\"ormander metric if the following three conditions are satisfied:
\begin{itemize}
\item[$(i)$] (slow variation) there exist $C\geq 1$ and $r>0$ such that for all $X,Y,T\in W$
    \beqs
    g_X(X-Y)\leq r^2\Rightarrow C^{-1}g_Y(T)\leq g_X(T)\leq Cg_Y(T);
    \eeqs
\item[$(ii)$] (temperance) there exist $C\geq 1$, $N\in \NN$ such that for all $X,Y,T\in W$
    \beqs
    \left(g_X(T)/g_Y(T)\right)^{\pm 1}\leq C(1+g^{\sigma}_X(X-Y))^N;
    \eeqs
\item[$(iii)$] (the uncertainty principle) $g_X(T)\leq g^{\sigma}_X(T)$, for all $X,T\in W$.
\end{itemize}
We call $C$, $r$ and $N$ the structure constants of $g$. We say that $g$ is symplectic if $g=g^{\sigma}$. Denote $\lambda_g(X)=\inf_{T\in W\backslash\{0\}} (g^{\sigma}_X(T)/g_X(T))^{1/2}$; it is Borel measurable and $\lambda_g(X)\geq 1$, $\forall X\in W$. Given $Y\in W$ and $r>0$, denote $U_{Y,r}=\{X\in W|\, g_Y(X-Y)\leq r^2\}$ and define $\delta_r(X,Y)=1+g^{\sigma}_X\wedge g^{\sigma}_Y(U_{X,r}-U_{Y,r})$, $X,Y\in W$; where $g^{\sigma}_X\wedge g^{\sigma}_Y$ denotes the harmonic mean of the positive-definite quadratic forms $g^{\sigma}_X$ and $g^{\sigma}_Y$. The function $(X,Y)\mapsto \delta_r(X,Y)$ is Borel measurable on $W\times W$ and when $r\leq r'$ where $r'$ depends only on the structure constants of $g$, the function $\delta_r$ enjoys very useful properties; see \cite[Section 2.2.6]{lernerB} for the complete account.\\
A positive Borel measurable function $M$ on $W$ is said to be $g$-admissible if there are $C\geq 1$, $r>0$ and $N\in\NN$ such that for all $X,Y\in W$
\begin{gather*}
g_X(X-Y)\leq r^2\Rightarrow C^{-1}M(Y)\leq M(X)\leq CM(Y);\\
\left(M(X)/M(Y)\right)^{\pm1}\leq C(1+g^{\sigma}_X(X-Y))^N.
\end{gather*}
We denote by $g^{\#}_X$ the geometric mean of $g_X$ and $g^{\sigma}_X$: $g^{\#}_X=\sqrt{g_X\cdot g^{\sigma}_X}=\sqrt{g^{\sigma}_X\cdot g_X}$ (cf. \cite[Definition 4.4.26, p. 341]{lernerB}). Then $X\mapsto g^{\#}_X$ is a symplectic H\"ormander metric, called the symplectic intermediate of $g$, and every $g$-admissible weight is also $g^{\#}$-admissible (see \cite{toft} and \cite[Proposition 2.2.20, p. 78]{lernerB}); furthermore $g_X\leq g^{\#}_X\leq g^{\sigma}_X$.\\
\indent Given a $g$-admissible weight $M$, the space of symbols $S(M,g)$ is defined as the space of all $a\in\mathcal{C}^{\infty}(W)$ for which
\beq\label{cnbvlj135}
\|a\|^{(k)}_{S(M,g)}=\sup_{l\leq k}\sup_{\substack{X\in W\\ T_1,\ldots, T_l\in W\backslash\{0\}}}\frac{|a^{(l)}(X;T_1,\ldots,T_l)|} {M(X)\prod_{j=1}^lg_X(T_j)^{1/2}}<\infty,\,\,\, \forall k\in\NN.
\eeq
With this system of seminorms, $S(M,g)$ becomes an $(F)$-space. One can always regularise the metric making it to be smooth (hence Riemannian) without changing the notion of $g$-admissibility of a weight and the space $S(M,g)$; furthermore the same can be done for any $g$-admissible weight (see \cite{hormander}, \cite[Remark 2.2.8, p. 71]{lernerB}). In fact, given any $g$-admissible weight $M$, there exists a smooth $g$-admissible weight $\tilde{M}\in S(M,g)$ and $C>0$ such that $M(X)\leq C\tilde{M}(X)$, $\forall X\in W$. The definition of $S(M,g)$ can be naturally extended to matrix valued symbols. Namely, let $\tilde{V}$ be a finite dimensional complex Banach space (from now on, always abbreviated as $(B)$-space) with norm $\|\cdot\|_{\tilde{V}}$ and denote by $\|\cdot \|_{\mathcal{L}_b(\tilde{V})}$ the induced norm on $\mathcal{L}_b(\tilde{V})$. One defines the space of $\mathcal{L}_b(\tilde{V})$-valued symbols $S(M,g;\mathcal{L}_b(\tilde{V}))$ as the space of all $a\in\mathcal{C}^{\infty}(W;\mathcal{L}_b(\tilde{V}))$ for which $\|a\|^{(k)}_{S(M,g;\mathcal{L}_b(\tilde{V}))}<\infty$ where the latter norms are defined as in (\ref{cnbvlj135}) with $\|a^{(l)}(X;T_1,\ldots,T_l)\|_{\mathcal{L}_b(\tilde{V})}$ in place of $|a^{(l)}(X;T_1,\ldots,T_l)|$. Then $S(M,g;\mathcal{L}_b(\tilde{V}))=S(M,g)\otimes \mathcal{L}_b(\tilde{V})$ is an $(F)$-space (the topology on the tensor product is $\pi=\epsilon$ since $\mathcal{L}_b(\tilde{V})$ is finite dimensional).\\
\indent For any $a\in\SSS(W)$ (or $a\in\SSS(W;\mathcal{L}_b(\tilde{V}))$), the Weyl quantisation $a^w$ is the operator
\beqs
a^w\varphi(x)=\frac{1}{(2\pi)^n}\int_{V'}\int_V e^{i\langle x-y,\xi\rangle}a((x+y)/2,\xi)\varphi(y)dyd\xi,\,\,\, \varphi\in \SSS(V)\,\, (\mbox{resp.}\, \varphi\in\SSS(V;\tilde{V})),
\eeqs
where $dy$ is a left-right Haar measure on $V$ with $d\xi$ being its dual measure defined on $V'$ so that the Fourier inversion formula holds with the standard constants (consequently, $a^w$ as well as the product measure $dyd\xi$ on $W$ are unambiguously defined); $a^w$ extends to a continuous operator from $\SSS'(V)$ into $\SSS(V)$ (resp. from $\SSS'(V;\tilde{V})=\SSS'(V)\otimes \tilde{V}$ into $\SSS(V;\tilde{V})=\SSS(V)\otimes \tilde{V}$; the topology on the tensor product is $\pi=\epsilon$). The definition of the Weyl quantisation extends to symbols in $\SSS'(W)$ (resp. $\SSS'(W;\mathcal{L}_b(\tilde{V}))$) and in this case $a^w:\SSS(V)\rightarrow \SSS'(V)$ (resp. $a^w:\SSS(V;\tilde{V})\rightarrow \SSS'(V;\tilde{V})$) is continuous. When $a\in S(M,g)$ (resp. $a\in S(M,g;\tilde{L}_b(\tilde{V}))$), for $g$-admissible weight $M$, $a^w$ is in fact continuous as operator on $\SSS(V)$ (resp. $\SSS(V;\tilde{V})$) and it uniquely extends to an operator on $\SSS'(V)$ (resp, $\SSS'(V;\tilde{V})$) (cf. \cite{hormander}). Furthermore, if $a,b\in \SSS(W)$ (resp. $a,b\in\SSS(W;\mathcal{L}_b(\tilde{V}))$), then $a^wb^w=(a\#b)^w$, where $a\# b\in \SSS(W)$ (resp. $a\#b\in \SSS(W;\mathcal{L}_b(\tilde{V}))$) is given by
\beqs
a\#b(X)=\frac{1}{\pi^{2n}}\int_{W\times W} e^{-2i[X-Y_1,X-Y_2]}a(Y_1)b(Y_2)dY_1dY_2.
\eeqs
The bilinear map $\#$ extends uniquely to a weakly continuous bilinear map $S(M_1,g)\times S(M_2,g)\rightarrow S(M_1M_2,g)$ (in the sense of \cite[Theorem 4.2]{hormander}) and it is also continuous when these spaces are equipped with the $(F)$-topologies described above. This holds equally well in the $\mathcal{L}_b(\tilde{V})$-valued case (see \cite{hormander}).\\
Given $Y\in W$ and $r>0$, denote $U_{Y,r}=\{X\in W|\, g_Y(X-Y)\leq r^2\}$. We say that $a\in\mathcal{C}^{\infty}(W)$ is $g_Y$-confined in $U_{Y,r}$ (see \cite{bon-che,lernerB}) if
\beqs
\|a\|^{(k)}_{g_Y,U_{Y,r}}=\sup_{l\leq k} \sup_{\substack{X\in W\\ T_1,\ldots,T_l\in W\backslash\{0\}}} \frac{|a^{(l)}(X;T_1,\ldots,T_l)| (1+g^{\sigma}_Y(X-U_{Y,r}))^{k/2}}{\prod_{j=1}^l g_X(T_j)^{1/2}}<\infty,\,\,\, \forall k\in\NN.
\eeqs
We will use the same notations even when $a$ is $\mathcal{L}_b(\tilde{V})$-valued (of course, instead of the absolute value one uses $\|\cdot\|_{\mathcal{L}_b(\tilde{V})}$ in the above definition); from the context, it will always be clear whether we are considering scalar or $\mathcal{L}_b(\tilde{V})$-valued symbols. For fixed $Y$ and $r$, the set of $g_Y$-confined symbols in $U_{Y,r}$ coincides with $\SSS(W)$ (resp. with $\SSS(W;\mathcal{L}_b(\tilde{V}))$). A family $\SSS(W)\ni \varphi_Y$, $Y\in W$, (resp. $\SSS(W;\mathcal{L}_b(\tilde{V}))\ni \varphi_Y$, $Y\in W$) is said to be uniformly $g_Y$-confined in $U_{Y,r}$ if $\sup_{Y\in W} \|\varphi_Y\|^{(k)}_{g_Y,U_{Y,r}}<\infty$, $\forall k\in\NN$. There is $r_0>0$ which depends only on the structure constants of $g$ such that for each $r\leq r_0$ there is a smooth uniformly $g_Y$-confined family in $U_{Y,r}$ $Y\mapsto \varphi_Y$, $W\rightarrow \SSS(W)$, such that $\operatorname{supp}\varphi_Y\subseteq U_{Y,r}$, $\varphi_Y\geq 0$ and
\beq\label{vstnlp135}
\int_W \varphi_Y(X)|g_Y|^{1/2}dY=1,\,\,\, \forall X\in W,
\eeq
where $|g_Y|=\det g_Y$ (see \cite[Theorem 2.2.7, p. 70]{lernerB}). Given $a_j\in S(M_j,g)$ (resp. $a_j\in S(M_j,g;\mathcal{L}_b(\tilde{V}))$), $j=1,2$, and denoting $a_{j,Y}=a_j\varphi_Y$, $Y\in W$, it holds
\beqs
a_1\#a_2(X)=\int_{W\times W} a_{1,Y_1}\#a_{2,Y_2}(X)|g_{Y_1}|^{1/2} |g_{Y_2}|^{1/2}dY_1dY_2,\,\,\,\forall X\in W
\eeqs
(cf. the proof of \cite[Theorem 2.3.7, p. 91]{lernerB}). Furthermore, given $a\in S(M,g)$ (resp. $a\in S(M,g;\mathcal{L}_b(\tilde{V}))$), and denoting as before $a_Y=a\varphi_Y$ we have
\beqs
a^wu=\int_W a_Y^wu |g_Y|^{1/2}dY,\,\, u\in\SSS(V)\,\, (\mbox{resp.}\, u\in\SSS(V;\tilde{V})),
\eeqs
where the equality holds if we interpret the integral in Bochner sense as well as pointwise. Furthermore, for $\varphi_Y$, $Y\in W$, as above and any $r'>r$ there exist two strongly Borel measurable uniformly $g_Y$-confined families in $U_{Y,r'}$, $Y\mapsto\psi_Y$, $Y\mapsto \theta_Y$, $W\rightarrow \SSS(W)$, such that $\varphi_Y=\psi_Y\#\theta_Y$, $Y\in W$ (see \cite[Theorem 2.3.15, p. 98]{lernerB}). The Sobolev space $H(M,g)$, with a $g$-admissible weight $M$, is the space of all $u\in\SSS'(V)$ such that
\beq\label{vcnjep135}
\int_W M(Y)^2\|\theta_Y^wu\|^2_{L^2(V)}|g_Y|^{1/2}dY<\infty.
\eeq
It is a Hilbert space with inner product
\beq\label{vdklbr135}
(u,v)_{H(M,g)}=\int_W M(Y)^2 (\theta_Y^wu,\theta_Y^wv)_{L^2(V)}|g_Y|^{1/2} dY
\eeq
and its definition and topology do not depend on the choice of the partition of unity $\varphi_Y$, $Y\in W$, and the families $\psi_Y$, $\theta_Y$, $Y\in W$. The space $\SSS(V)$ is continuously and densely included into $H(M,g)$ and the latter is continuously and densely included into $\SSS'(V)$. If $a\in S(M',g)$, $a^w$ restricts to a continuous operator from $H(M,g)$ into $H(M/M',g)$; in particular, if $M_1/M_2$ is bounded from below then $H(M_1,g)$ is continuously (and densely) included into $H(M_2,g)$. Furthermore, $H(1,g)$ is just $L^2(V)$. (We refer to \cite[Section 2.6]{lernerB} and \cite{bon-che} for the proofs of these properties of the Sobolev spaces $H(M,g)$.) The definition of $H(M,g;\tilde{V})$ is similar: $u\in\SSS'(V;\tilde{V})$ is in $H(M,g;\tilde{V})$ if the quantity (\ref{vcnjep135}) is finite with $\|\theta_Y^wu\|_{\mathcal{L}^2(V)}$ replaced by $\|(\theta_YI)^wu\|_{L^2(V;\tilde{V})}$, where $I:\tilde{V}\rightarrow \tilde{V}$ is the identity operator. It is a $(B)$-space since it is topologically isomorphic to $H(M,g)\otimes \tilde{V}$. Fixing an inner product on $\tilde{V}$ naturally induces an inner product on $L^2(V;\tilde{V})$ which, in turn, induces an inner product on $H(M,g;\tilde{V})$ (similarly as in (\ref{vdklbr135})) and the latter becomes a Hilbert space. Moreover, the above isomorphism verifies that all facts we mentioned for the scalar valued case remain true in the vector-valued case as well.\\
\indent For any $A\in\mathcal{L}(\SSS(V),\SSS'(V))$ and any linear form $L$ on $W$, we denote by $\ad L^w\cdot A$ the commutator of $L^w$ and $A$, i.e. $\ad L^w\cdot A=L^wA-AL^w\in\mathcal{L}(\SSS(V),\SSS'(V))$. When $LX=[T,X]$, for some $T\in W$, it will be convenient to identify the linear form $L$ with $T$. If $a\in S(M,g)$, the following seminorms are always finite for all families $\phi_Y$, $Y\in W$, which are uniformly $g_Y$-confined in $U_{Y,r}$
\beqs
\|a^w\|^{(k)}_{op(M,g)}=\sup_{Y\in W} \sup_{\substack{l\leq k\\ g_Y(L_1)\leq 1,\ldots, g_Y(L_l)\leq 1}} M(Y)^{-1}\|\ad L_1^w\ldots \ad L_l^w\cdot \phi_Y^w a^w\|_{\mathcal{L}(L^2)}<\infty,\,\,\, \forall k\in\NN,
\eeqs
where $L_jX=[T_j,X]$ and, as mentioned above, we identified $L_j$ with $T_j$. In fact, a result of Bony and Chemin \cite[Theorem 5.5]{bon-che} (see also \cite[Theorem 2.6.12, p. 145]{lernerB}) proves that the converse is also true. Namely, if $A\in\mathcal{L}(\SSS(V),\SSS'(V))$ is such that for all families $\phi_Y$, $Y\in W$, which are uniformly $g_Y$-confined in $U_{Y,r}$, $\ad L_1^w\ldots \ad L_k^w\cdot \phi_Y^w A\in \mathcal{L}(L^2)$, $\forall Y\in W$, $\forall k\in\NN$, and the seminorms $\|A\|^{(k)}_{op(M,g)}$ are finite for all $k\in\NN$, then $A=a^w$, for some $a\in S(M,g)$. In fact, with $\varphi_Y$, $\psi_Y$, $\theta_Y$, $Y\in W$, as before, one needs to check this only for the uniformly confined family $\theta_Y$, $Y\in W$, and
\beqs
\forall k\in\NN,\, \exists C>0,\, \exists l\in\NN,\,\,\, \|a\|^{(k)}_{S(M,g)}\leq C\|a^w\|^{(l)}_{op(M,g)},
\eeqs
with $\|a^w\|^{(l)}_{op(M,g)}$ defined via $\theta_Y$, $Y\in W$. All of the above hold equally well in the vector-valued case with $\phi_Y^w$ and $\ad L^w$ replaces by $(\phi_YI)^w$ and $\ad (LI)^w$ respectively; in fact, the validity of these results is a direct consequence of the topological isomorphism $S(M,g;\mathcal{L}_b(\tilde{V}))\cong S(M,g)\otimes \mathcal{L}_b(\tilde{V})$.\\
\indent On a couple of occasions we will impose the following additional assumption on $g$; we will always emphasise when we assume it. We say the H\"ormander metric $g$ is geodesically temperate if there exist $C\geq1$ and $N\in\NN$ such that
\beq\label{cntepm135}
g_X(T)\leq Cg_Y(T)(1+d(X,Y))^N,\,\,\, \forall X,Y,T\in W,
\eeq
where $d(\cdot,\cdot)$ stands for the geodesic distance on $W$ induced by $g^{\#}$. A number of metrics which correspond to different calculi are geodesically temperate: the $S^m_{\rho,\delta}$-calculus, the semi-classical, the Shubin calculus (see \cite[Example 7.3]{bon-che}, \cite[Lemmas 2.6.22 and 2.6.23, p. 154]{lernerB}). In fact, \cite[Theorem 5 $(i)$]{bony} proves that if the positive Borel measurable functions $\varphi$ and $\Phi$ on $\RR^{2n}$ are such that
\beqs
g_{x,\xi}=\varphi(x,\xi)^{-2}|dx|^2+\Phi(x,\xi)^{-2}|d\xi|^2
\eeqs
is a H\"ormander metric than $g$ satisfies (\ref{cntepm135}) with $d(\cdot,\cdot)$ standing for the geodesic distance induced by $g^{\sigma}$. Applying this result to $g^{\#}_{x,\xi}=\Phi\varphi^{-1}|dx|^2+\varphi\Phi^{-1}|d\xi|^2$ we conclude the latter is geodesically temperate. As $g=\varphi^{-1}\Phi^{-1}g^{\#}$, \cite[Lemma 2.6.22, p. 154]{lernerB} verifies that $g$ is also geodesically temperate ($\varphi\Phi\geq 1$ since $g$ is a H\"ormander metric). In particular, the geodesic temperance is valid for the Beals-Fefferman calculus \cite{bf1,beals1,beals2} (cf. \cite[Example 3]{hormander}) as well as the Nicola-Rodino calculus \cite{NR}.

\section{Inverse smoothness in $S(1,g;\mathcal{L}_b(\tilde{V}))$}\label{vjtknc159}

The result of Bony and Chemin \cite[Theorem 7.6]{bon-che} (see also \cite[Theorem 2.6.27, p. 158]{lernerB}) verifies that the Weyl-H\"ormander calculus is spectrally invariant provided the H\"ormnader metric $g$ is geodesically temperate. That is, given $a\in S(1,g)$ such that $a^w$ is invertible on $L^2(V)$ its inverse is pseudodifferential operator with symbol in $S(1,g)$ (i.e. the operators with symbols in $S(1,g)$ form a $\Psi^*$-algebra in the $C^*$-algebra $\mathcal{L}_b(L^2(V))$; cf. \cite{gramsch,schrohe}). In this section, we prove that this process of taking inverses preserves the regularity in the following sense. If $\lambda\mapsto a_{\lambda}$ is of class $\mathcal{C}^N$, $0\leq N\leq\infty$, with values in $S(1,g;\mathcal{L}_b(\tilde{V}))$ such that $a_{\lambda}^w$ is invertible in $\mathcal{L}_b(L^2(V;\tilde{V}))$, then the mapping $\lambda\mapsto b_{\lambda}$, where $b_{\lambda}^w$ is the inverse of $a_{\lambda}^w$, is also of class $\mathcal{C}^N$.\\
\indent Before we state and prove this result, we need the following technical results.
They have the same form as  \cite[Lemma 2.6.25, p. 155]{lernerB}, see also \cite[Lemma 2.6.26, p. 156]{lernerB} and Bony and Chemin \cite[Lemma 7.4]{bon-che}  \cite[Lemma 7.5]{bon-che} but the right hand sides of the estimates in our paper have  slightly more precise forms. Moreover, the second lemma is slightly more general variant of the second cited lemma.

\begin{lemma}\label{lemmaonprodsub}
Let $g$ be a H\"ormander metric. Then $\forall N_0\geq0$, $\exists C_0>0$, $\exists k_0\in \NN$, $\forall N_1\geq0$, $\exists C_1>0$, $\exists k_1\in \NN$, $\forall \nu\in\ZZ_+$, $\forall J\subseteq \{0,\ldots,\nu-1\}$, $J\neq \emptyset$, $\forall c_0,\ldots,c_{\nu}\in\SSS(W;\mathcal{L}_b(\tilde{V}))$, $\forall Y_0,\ldots,Y_{\nu}\in W$ it holds
\begin{multline*}
\|c_0^w\ldots c_{\nu}^w\|_{\mathcal{L}_b(L^2(V;\tilde{V}))}\leq C_0^{\nu-|J|}C_1^{1+|J|}\|c_0\|^{(k_1)}_{g_{Y_0},U_{Y_0,r}} \|c_{\nu}\|^{(k_1)}_{g_{Y_{\nu}},U_{Y_{\nu},r}} \left(\max_{j\in K'} \|c_j\|^{(k_0)}_{g_{Y_j},U_{Y_j},r}\right)^{\nu-|J\cup\{0,\nu-1\}|}\\
\cdot\left(\max_{j\in K} \|c_j\|^{(k_1)}_{g_{Y_j},U_{Y_j},r}\right)^{|(J\cup\{0\})\backslash\{\nu-1\}|}\,\, \prod_{j=0}^{\nu-1}\delta_r(Y_j,Y_{j+1})^{-N_0} \prod_{j\in J}\delta_r(Y_j,Y_{j+1})^{-N_1},
\end{multline*}
with $K=(J\cup (J+1))\backslash\{0,\nu\}$ and $K'=(\NN\cap[1,\nu-1])\backslash(J\cap(J+1))$. If $K=\emptyset$ then $|(J\cup\{0\})\backslash\{\nu-1\}|=0$ and we define
\beqs
\left(\max_{j\in K} \|c_j\|^{(k_1)}_{g_{Y_j},U_{Y_j},r}\right)^{|(J\cup\{0\})\backslash\{\nu-1\}|}=1;
\eeqs
if $K'=\emptyset$ then $\nu-|J\cup\{0,\nu-1\}|=0$ and we define
\beqs
\left(\max_{j\in K'} \|c_j\|^{(k_0)}_{g_{Y_j},U_{Y_j},r}\right)^{\nu-|J\cup\{0,\nu-1\}|}=1.
\eeqs
\end{lemma}

\begin{proof} Applying the same technique as in the proof of \cite[Lemma 7.4]{bon-che} (see also the proof of \cite[Lemma 2.6.25, p. 155]{lernerB}) we infer\footnote{here and throughout the article we use the principle of vacuous (empty) product; i.e. $\prod_{j=1}^0r_j=1$}
\begin{align}
\|&c_0^w\ldots c_{\nu}^w\|^2_{\mathcal{L}_b(L^2(V;\tilde{V}))}\leq
C_0^{2\nu-2|J|}C_1^{2+2|J|}\left(\|c_0\|^{(k_1)}_{g_{Y_0},U_{Y_0,r}}\right)^2 \left(\|c_{\nu}\|^{(k_1)}_{g_{Y_{\nu}},U_{Y_{\nu},r}}\right)^2 \nonumber\\
&\cdot\left(\prod_{j\in J\backslash\{0\}} \|c_j\|^{(k_1)}_{g_{Y_j},U_{Y_j},r}\right)\left(\prod_{j\in (J+1)\backslash\{\nu\}} \|c_j\|^{(k_1)}_{g_{Y_j},U_{Y_j},r}\right)
\left(\prod_{\substack{j=1\\ j\not\in J}}^{\nu-1} \|c_j\|^{(k_0)}_{g_{Y_j},U_{Y_j},r}\right) \left(\prod_{\substack{j=1\\ j\not\in J+1}}^{\nu-1}\|c_j\|^{(k_0)}_{g_{Y_j},U_{Y_j},r}\right)\label{vsntek157}\\
&\cdot\left(\prod_{j=0}^{\nu-1}\delta_r(Y_j,Y_{j+1})^{-2N_0}\right) \left(\prod_{j\in J}\delta_r(Y_j,Y_{j+1})^{-2N_1}\right),\nonumber
\end{align}
with $C_0$ and $k_0\geq 2n+1$ depending only on $N_0$ and $C_1\geq C_0$ and $k_1\geq k_0$ depending on $N_0+N_1$. Now, one can deduce the claim in the lemma by considering the four cases depending on whether $0$ or $\nu-1$ belongs to $J$ or not. We illustrate the main ideas on the case when $0\not\in J$ and $\nu-1\not\in J$. As $J\neq \emptyset$, $\nu\geq 3$. Denote $s=\min\{j|\, j\in J\}$, $t=\max\{j|\, j\in J\}$. Clearly $1\leq s\leq t\leq\nu-2$; $K,K'\neq\emptyset$; $t+1\not\in J$; $s\not\in J+1$. Denote $\tilde{c}=\max_{j\in K} \|c_j\|^{(k_1)}_{g_{Y_j},U_{Y_j},r}$, $\tilde{\tilde{c}}=\max_{j\in K'} \|c_j\|^{(k_0)}_{g_{Y_j},U_{Y_j},r}$. The products in (\ref{vsntek157}) are equal to
\begin{align*}
&\left(\prod_{j\in J\backslash\{0\}} \|c_j\|^{(k_1)}_{g_{Y_j},U_{Y_j},r}\right)\left(\prod_{j\in (J+1)\backslash\{\nu\}} \|c_j\|^{(k_1)}_{g_{Y_j},U_{Y_j},r}\right)
\left(\prod_{\substack{j=1\\ j\neq t+1,\,j\not\in J}}^{\nu-1} \|c_j\|^{(k_0)}_{g_{Y_j},U_{Y_j},r}\right)\\
&\qquad\qquad\cdot\left(\prod_{\substack{j=1\\ j\neq s,\, j\not\in J+1}}^{\nu-1}\|c_j\|^{(k_0)}_{g_{Y_j},U_{Y_j},r}\right) \|c_{t+1}\|^{(k_0)}_{g_{Y_{t+1}},U_{Y_{t+1}},r} \|c_s\|^{(k_0)}_{g_{Y_s},U_{Y_s},r}\\
&\quad\leq \tilde{c}^{2|J|} \tilde{\tilde{c}}^{2(\nu-|J|-2)} \|c_{t+1}\|^{(k_0)}_{g_{Y_{t+1}},U_{Y_{t+1}},r} \|c_s\|^{(k_0)}_{g_{Y_s},U_{Y_s},r}.
\end{align*}
As $t+1,s\in K$ and $k_0\leq k_1$, we infer $\|c_s\|^{(k_0)}_{g_{Y_s},U_{Y_s},r}\leq \tilde{c}$, $\|c_{t+1}\|^{(k_0)}_{g_{Y_{t+1}},U_{Y_{t+1}},r}\leq \tilde{c}$ and the above is bounded by $\tilde{c}^{2(|J|+1)} \tilde{\tilde{c}}^{2(\nu-|J|-2)}$. As $|J|+1=|(J\cup\{0\})\backslash\{\nu-1\}|$ and $\nu-|J|-2=\nu-|J\cup\{0,\nu-1\}|$, we deduce the claim in the lemma.
\end{proof}

The following is a slight generalisation of \cite[Lemma 7.5]{bon-che} (cf. \cite[Lemma 2.6.26, p. 156]{lernerB}).

\begin{lemma}\label{estfortheprodlem}
Let $g$ be a geodesically temperate symplectic H\"ormader metric. There exist $C_0>0$ and $k_0\in\ZZ_+$ which depend only on the structure constants of $g$, and for all $k\in\NN$, there exist $C_1,N_1>0$, $k_1\in\ZZ_+$, such that for $\nu\in\ZZ_+$ and $a_1,\ldots,a_{\nu}\in S(1,g;\mathcal{L}_b(\tilde{V}))$ it holds
\beqs
\|a_1^w\ldots a_{\nu}^w\|^{(k)}_{op(1,g)}\leq C_1 (\nu+1)^{N_1}\left(C_0\max_{j=1,\ldots,\nu}\|a_j\|^{(k_0)}_{S(1,g;\mathcal{L}_b(\tilde{V}))}\right)^{\nu-4k} \left(C_1\max_{j=1,\ldots,\nu}\|a_j\|^{(k_1)}_{S(1,g;\mathcal{L}_b(\tilde{V}))}\right)^{4k}.
\eeqs
\end{lemma}

\begin{proof} Let $\varphi_Y\in\SSS(W)$, $Y\in W$, be the decomposition of unity given in \cite[Theorem 2.2.7, p. 70]{lernerB}, that is $Y\mapsto \varphi_Y$, $W\rightarrow \SSS(W)$, is a smooth family of non-negative functions such that $\operatorname{supp}\varphi_Y\subseteq U_{Y,r}$ and (\ref{vstnlp135}) holds true. Denote $a_{j,Y}=\varphi_Ya_j$, $j=1,\ldots,\nu$, and set $a_{0,Y}=\theta_YI$ and $a_{\nu+1,Y}=\varphi_YI$, with $I:\tilde{V}\rightarrow \tilde{V}$ being the identity operator. Let $k\in\ZZ_+$. Fix $Y_0\in W$ and let $L_j(X)=[T_j,X]$, $j=1,\ldots,k$, with $g_{Y_0}(T_j)=1$. Employing the same technique as in the proof of \cite[Lemma 2.6.26, p. 156]{lernerB} (see also the proof of \cite[Lemma 7.5]{bon-che}) we deduce that $\|\ad (L_1I)^w\ldots\ad (L_kI)^w \cdot(\theta_{Y_0}I)^wa_1^w\ldots a_{\nu}^w\|_{\mathcal{L}_b(L^2(V;\tilde{V}))}$ is bounded by sum of $(\nu+2)^k$ terms $\tilde{\omega}_k$ of the form
\beq\label{estforomek}
\tilde{\omega}_k=\int_{W^{\nu+1}}\|b_0^w\ldots b_{\nu+1}^w\|_{\mathcal{L}_b(L^2(V;\tilde{V}))} |g_{Y_1}|^{1/2}\ldots |g_{Y_{\nu+1}}|^{1/2} dY_1\ldots dY_{\nu+1},
\eeq
where $b_j=(\prod_{\alpha\in E_j} \partial_{T_{\alpha}})a_{j,Y_j}$ and $E_j$, $j=0,\ldots,\nu+1$, are disjoint possibly empty subsets of $\{1,\ldots,k\}$.\footnote{we employ the convention $\prod_{s\in\emptyset} B_s=\mathrm{Id}$} Let $J=\{j\in\NN|\, E_j\neq \emptyset\}$; clearly $|J|\leq\sum_j |E_j|\leq k$. Similarly as in the proof of \cite[Lemma 2.6.26, p. 156]{lernerB}, we define $c_j=b_j=a_{j,Y_j}$ for $j\not\in J$, and $c_j=b_j(\prod_{\alpha\in E_j} g_{Y_j}(T_{\alpha})^{-1/2})$ for $j\in J$. Employing the geodesic temperance of $g(=g^{\#})$, in an analogous fashion as in the proof of the quoted result, we infer
\beqs
g_{Y_j}(T_{\alpha})\leq C(\nu+1)^{N-1} \sum_{l=0}^{j-1}\delta_r(Y_l,Y_{l+1})^{N^2},\,\,\, j=1,\ldots,\nu+1,\, \alpha=1,\ldots,k,
\eeqs
where $C$ and $N$ depend only on the structure constants of $g$ and the constants in (\ref{cntepm135}). If $J\backslash\{0\}\neq \emptyset$, we infer
\beq
\|b_0^w\ldots b_{\nu+1}^w\|_{\mathcal{L}_b(L^2(V;\tilde{V}))}&\leq& C^k(\nu+1)^{(N-1)k}\|c_0^w\ldots c_{\nu+1}^w\|_{\mathcal{L}(L^2_b(V;\tilde{V}))}\prod_{j\in J\backslash\{0\}}\left(\sum_{l=0}^{j-1} \delta_r(Y_l,Y_{l+1})^{N^2}\right)^{|E_j|}\nonumber\\
&=& C^k(\nu+1)^{(N-1)k}\|c_0^w\ldots c_{\nu+1}^w\|_{\mathcal{L}(L^2_b(V;\tilde{V}))}\sum_{\mu} F_{\mu},\label{vclpke135}
\eeq
where, the very last sum has at most $(\nu+1)^k$ terms and each $F_{\mu}$ is of the following form
\beqs
F_{\mu}=\prod_{j=0}^{\nu}\delta_r(Y_j,Y_{j+1})^{m_{j,\mu}N^2}
\eeqs
with $m_{j,\mu}\in\NN$ satisfying $\sum_{j=0}^{\nu}m_{j,\mu}\leq\sum_{j\in J}|E_j|\leq k$. If $J=\{0\}$ (i.e. only $E_0$ is non-empty) then (\ref{vclpke135}) remains true if the sum over $\mu$ has only one term $F_{\mu}=1$, i.e. $m_{j,\mu}=0$, $j=0,\ldots,\nu$. For each $\mu$, let $J_{\mu}=\{j\in\NN|\, m_{j,\mu}\neq0\}$; clearly $|J_{\mu}|\leq \sum_j m_{j,\mu}\leq k$. Define
\beqs
\tilde{F}_{\mu}=F_{\mu}\prod_{\substack{j\in ((J\cup\{\nu\})\backslash\{\nu+1\})\cup ((J-1)\cap\NN)\\ j\not\in J_{\mu}}} \delta_r(Y_j,Y_{j+1})^{N^2}= \prod_{j=0}^{\nu} \delta_r(Y_j,Y_{j+1})^{\tilde{m}_{j,\mu}N^2}.
\eeqs
Then $\sum_{j=0}^{\nu}\tilde{m}_{j,\mu}\leq 3k+1$. Let $\tilde{J}_{\mu}=\{j\in\NN|\, \tilde{m}_{j,\mu}\neq0\}$. Then $|\tilde{J}_{\mu}|\leq 3k+1$, $\nu\in\tilde{J}_{\mu}$ and
\beq
\|b_0^w\ldots b_{\nu+1}^w\|_{\mathcal{L}(L^2(V;\tilde{V}))}&\leq& C^k(\nu+1)^{(N-1)k}\sum_{\mu}\|c_0^w\ldots c_{\nu+1}^w\|_{\mathcal{L}(L^2(V;\tilde{V}))}\tilde{F}_{\mu};\label{nrtscl135}\\
J\backslash\{0,\nu+1\}&\subseteq& \tilde{J}_{\mu}\cap(\tilde{J}_{\mu}+1).\label{kstvcp157}
\eeq
For each $\mu$ we apply Lemma \ref{lemmaonprodsub} with $N_0\geq 0$ such that $\sup_{Y\in W}\int_{W}\delta_r(Y,Z)^{-N_0}|g_Z|^{1/2}dZ <\infty$, $N_1=kN^2$ and $\tilde{J}_{\mu}\subseteq\{0,\ldots,\nu\}$ (with $\nu+1$ in place of $\nu$) to obtain
\begin{multline*}
\|c_0^w\ldots c_{\nu+1}^w\|_{\mathcal{L}(L^2(V;\tilde{V}))}\leq C_0^{\nu+1-|\tilde{J}_{\mu}|}C_1^{1+|\tilde{J}_{\mu}|} \|c_0\|^{(k_1)}_{g_{Y_0},U_{Y_0,r}} \|c_{\nu+1}\|^{(k_1)}_{g_{Y_{\nu+1}},U_{Y_{\nu+1},r}}\\
\cdot\left(\max_{j\in \tilde{K}'_{\mu}} \|c_j\|^{(k_0)}_{g_{Y_j},U_{Y_j},r}\right)^{\nu+1-|\tilde{J}_{\mu}\cup\{0,\nu\}|} \left(\max_{j\in \tilde{K}_{\mu}} \|c_j\|^{(k_1)}_{g_{Y_j},U_{Y_j},r}\right)^{|(\tilde{J}_{\mu}\cup\{0\})\backslash\{\nu\}|}\\
\cdot\prod_{j=0}^{\nu}\delta_r(Y_j,Y_{j+1})^{-N_0} \prod_{j\in \tilde{J}_{\mu}}\delta_r(Y_j,Y_{j+1})^{-kN^2},
\end{multline*}
where $\tilde{K}_{\mu}=(\tilde{J}_{\mu}\cup (\tilde{J}_{\mu}+1))\backslash\{0,\nu+1\}\neq\emptyset$ (since $\nu\in\tilde{J}_{\mu}$) and $\tilde{K}'_{\mu}=(\NN\cap[1,\nu])\backslash (\tilde{J}_{\mu}\cap(\tilde{J}_{\mu}+1))$; of course, we may assume $k_0\leq k_1$. Notice that $\tilde{K}'_{\mu}=\emptyset$ if and only if $\tilde{J}_{\mu}=\{0,\ldots,\nu\}$. By construction, there exists $C'\geq 1$, which depends only on the structure constants of $g$, such that $\|c_j\|^{(l)}_{g_{Y_j},U_{Y_j},r}\leq C'^k\|a_{j,Y_j}\|^{(l+k)}_{g_{Y_j},U_{Y_j},r}$, for all $l\in\NN$, $j=1,\ldots,\nu$. Furthermore, if $j\in \tilde{K}'_{\mu}$, (\ref{kstvcp157}) implies $c_j= a_{j,Y_j}$. Finally, notice that the seminorms of $c_0$ and $c_{\nu+1}$ depend only on the structure constants of $g$ (recall $a_{0,Y_0}=\theta_{Y_0}I$, $a_{\nu+1,Y}=\varphi_YI$). Plugging these estimates in (\ref{nrtscl135}) we infer (since $\nu\in\tilde{J}_{\mu}$)
\begin{multline}\label{estforbs}
\|b_0^w\ldots b_{\nu+1}^w\|_{\mathcal{L}(L^2(V;\tilde{V}))}\leq C'_1(\nu+1)^{(N-1)k}\prod_{j=0}^{\nu}\delta_r(Y_j,Y_{j+1})^{-N_0}\\
\cdot \sum_{\mu} \left(C'_0\max_{j=1\ldots,\nu} \|a_j\|^{(k_0)}_{S(1,g;\mathcal{L}_b(\tilde{V}))}\right)^{\nu+1-|\tilde{J}_{\mu}\cup\{0\}|} \left(C'_1\max_{j=1,\ldots,\nu} \|a_j\|^{(k_1+k)}_{S(1,g;\mathcal{L}_b(\tilde{V}))}\right)^{|\tilde{J}_{\mu}\cup\{0\}|-1},
\end{multline}
with $C'_0$ depending only on the structure constants of $g$ and $C'_1$ independent of $\nu$ and $a_1,\ldots,a_{\nu}$. Since $|\tilde{J}_{\mu}\cup\{0\}|\leq 3k+2\leq 4k+1$ we deduce
\begin{multline*}
\left(C'_0\max_{j=1\ldots,\nu} \|a_j\|^{(k_0)}_{S(1,g;\mathcal{L}_b(\tilde{V}))}\right)^{\nu+1-|\tilde{J}_{\mu}\cup\{0\}|} \left(C'_1\max_{j=1,\ldots,\nu} \|a_j\|^{(k_1+k)}_{S(1,g;\mathcal{L}_b(\tilde{V}))}\right)^{|\tilde{J}_{\mu}\cup\{0\}|-1}\\
\leq\left(C'_0\max_{j=1\ldots,\nu} \|a_j\|^{(k_0)}_{S(1,g;\mathcal{L}_b(\tilde{V}))}\right)^{\nu-4k} \left(C'_1\max_{j=1,\ldots,\nu} \|a_j\|^{(k_1+k)}_{S(1,g;\mathcal{L}_b(\tilde{V}))}\right)^{4k}.
\end{multline*}
Having in mind the latter and the fact that the sum over $\mu$ has at most $(\nu+1)^k$ terms, we can employ the estimate (\ref{estforbs}) in (\ref{estforomek}) to conclude the claim in the lemma; the estimates for $\|(\theta_{Y_0}I)^wa_1^w\ldots a_{\nu}^w\|_{\mathcal{L}_b(L^2(V;\tilde{V}))}$ (when $k=0$) can be obtained in analogous fashion as for the case when $k\in\ZZ_+$.
\end{proof}

The following remarks will prove useful throughout the rest of the article; we state them here for convenience but we will frequently tacitly apply them.

\begin{remark}
If $E_j$, $j=1,2$, are two locally compact Hausdorff topological spaces and $\mathbf{f}_j:E_j\rightarrow \SSS(W;\mathcal{L}_b(\tilde{V}))$, $j=1,2$, continuous mappings, then the mapping
\beq\label{contmap}
(\lambda,\mu)\mapsto \mathbf{f}_1(\lambda)\#\mathbf{f}_2(\mu),\,\,\, E_1\times E_2\rightarrow \SSS(W;\mathcal{L}_b(\tilde{V})),
\eeq
is continuous; this is a direct consequence of \cite[Corollary 2.3.3, p. 85]{lernerB}. Consequently, if $E_1=E_2=E$, the mapping $\lambda\mapsto \mathbf{f}_1(\lambda)\# \mathbf{f}_2(\lambda)$, $E\rightarrow \SSS(W;\mathcal{L}_b(\tilde{V}))$, is continuous.\\
\indent If $E_j$, $j=1,2$, are as above and $\mathbf{f}_j:E_j\rightarrow S(M_j,g;\mathcal{L}_b(\tilde{V}))$, $j=1,2$, are continuous mappings where $M_1$ and $M_2$ are admissible weights for $g$ then \cite[Theorem 2.3.7, p. 91]{lernerB} verifies that the mapping
\beq\label{conwithvalinss}
(\lambda,\mu)\mapsto \mathbf{f}_1(\lambda)\#\mathbf{f}_2(\mu),\,\,\, E_1\times E_2\rightarrow S(M_1M_2,g;\mathcal{L}_b(\tilde{V}))
\eeq
is continuous. Again, if $E_1=E_2=E$, this implies that the mapping $\lambda\mapsto \mathbf{f}_1(\lambda)\# \mathbf{f}_2(\lambda)$, $E\rightarrow S(M_1M_2,g;\mathcal{L}_b(\tilde{V}))$, is also continuous.
\end{remark}

\begin{remark}
If $E_j$, $j=1,2$, are two smooth manifolds without boundary (we always assume the smooth manifolds are paracompact) and $\mathbf{f}_j:E_j\rightarrow \SSS(W;\mathcal{L}_b(\tilde{V}))$, $j=1,2$, are of class $\mathcal{C}^N$, $0\leq N\leq\infty$, then so is the map (\ref{contmap}); this can be easily derived from \cite[Corollary 2.3.3, p. 85]{lernerB} (in fact, since the problem is local in nature, it is enough to prove it when $E_1$ and $E_2$ are Euclidean spaces). If $X_j$ is a smooth vector field on $E_j$ and $\mathbf{f}_j$ is smooth, $j=1,2$, then
\beq\label{smoothmappi}
X_1\times X_2(\mathbf{f}_1(\lambda)\#\mathbf{f}_2(\mu))= X_1\mathbf{f}_1(\lambda)\#\mathbf{f}_2(\mu)+ \mathbf{f}_1(\lambda)\# X_2\mathbf{f}_2(\mu)
\eeq
(of course, $X_1\mathbf{f}_1(\lambda)$ and $X_2\mathbf{f}_2(\mu)$ are smooth maps into $\SSS(W;\mathcal{L}_b(\tilde{V}))$). Consequently, if $E_1=E_2=E$, the mapping $\lambda\mapsto \mathbf{f}_1(\lambda)\#\mathbf{f}_2(\lambda)$, $E\rightarrow \SSS(W;\mathcal{L}_b(\tilde{V}))$, is smooth and the smooth vector fields on $E$ are derivations of the algebra $\mathcal{C}^{\infty}(E;\SSS(W;\mathcal{L}_b(\tilde{V})))$ (with the associative product $\#$). If $\mathbf{f}_j$, $j=1,2$, are of class $\mathbf{C}^N$, $N\geq 1$, and $X$ a smooth vector field on $E$, we still have
\beq\label{vckntp137}
X(\mathbf{f}_1(\lambda)\#\mathbf{f}_2(\lambda))=X\mathbf{f}_1(\lambda)\#\mathbf{f}_2(\lambda)+\mathbf{f}_1(\lambda)\# X\mathbf{f}_2(\lambda).
\eeq
\indent If $E_1$ and $E_2$ are as above and $\mathbf{f}_j:E_j\rightarrow S(M_j,g;\mathcal{L}_b(\tilde{V}))$, $j=1,2$, are of class $\mathcal{C}^N$, $0\leq N\leq\infty$, where $M_j$, $j=1,2$, are admissible weights for $g$, then the mapping (\ref{conwithvalinss}) is also of class $\mathcal{C}^N$ (by \cite[Theorem 2.3.7, p. 91]{lernerB}). When $\mathbf{f}_j$, $j=1,2$, are smooth, (\ref{smoothmappi}) holds true; in particular, if $E_1=E_2=E$ and $M_1=M_2=1$, the smooth vector fields on $E$ are derivations of the unital algebra $\mathcal{C}^{\infty}(E;S(1,g;\mathcal{L}_b(\tilde{V})))$ (with the associative product $\#$). Furthermore, if $\mathbf{f}_1$ and $\mathbf{f}_2$ are of class $\mathcal{C}^N$, $N\geq 1$, and $X$ a smooth vector field on $E$, (\ref{vckntp137}) remains valid.
\end{remark}

The main result of the section is the following.

\begin{theorem}\label{invconsmoot}
Assume that $g$ is a geodesically temperate H\"ormander metric. Let $E$ be a Hausdorff topological space and $\mathbf{f}:E\rightarrow S(1,g;\mathcal{L}_b(\tilde{V}))$ a continuous mapping. If for each $\lambda\in E$, $\mathbf{f}(\lambda)^w$ is invertible operator on $L^2(V;\tilde{V})$, then there exists a unique continuous mapping $\tilde{\mathbf{f}}:E\rightarrow S(1,g;\mathcal{L}_b(\tilde{V}))$ such that
\beq\label{invconmap}
\tilde{\mathbf{f}}(\lambda)\#\mathbf{f}(\lambda)=\mathbf{f}(\lambda)\#\tilde{\mathbf{f}}(\lambda)=I,\,\,\, \forall\lambda\in E.
\eeq
If $E$ is a smooth manifold without boundary and $\mathbf{f}:E\rightarrow S(1,g;\mathcal{L}_b(\tilde{V}))$ is of class $\mathcal{C}^N$, $0\leq N\leq \infty$, then $\tilde{\mathbf{f}}:E\rightarrow S(1,g;\mathcal{L}_b(\tilde{V}))$ is also of class $\mathcal{C}^N$.
\end{theorem}

\begin{proof} The existence of $\tilde{\mathbf{f}}:E\rightarrow S(1,g;\mathcal{L}_b(\tilde{V}))$ which satisfies (\ref{invconmap}) is a direct consequence of \cite[Theorem 2.6.27, p. 158]{lernerB}\footnote{this result is given only for the scalar valued case, but one can easily verify that the same proof works in the vector valued case as well} and the uniqueness easily follows from the fact that $S(1,g;\mathcal{L}_b(\tilde{V}))$ is a unital associative algebra. We need to proof the continuity and the fact that $\tilde{\mathbf{f}}$ is of class $\mathcal{C}^N$, respectively. Throughout the proof, we fix an inner product on $\tilde{V}$ and denote by $\|\cdot\|_{\tilde{V}}$ and $\|\cdot\|_{\mathcal{L}_b(\tilde{V})}$ the induced norms.\\
\indent The continuity of $\tilde{\mathbf{f}}$ follows from general facts on Fr\'echet algebras because of the following. The set of invertible elements of the Banach algebra $\mathcal{L}_b(L^2(V;\tilde{V}))$ is open and thus its inverse image under the continuous mapping  $S(1,g;\mathcal{L}_b(\tilde{V}))\rightarrow \mathcal{L}_b(L^2(V;\tilde{V}))$, $a\mapsto a^w$, is open in $S(1,g;\mathcal{L}_b(V;\tilde{V}))$ and it coincides with the set of invertible elements of $S(1,g;\mathcal{L}_b(\tilde{V}))$ because of spectral invariance \cite[Theorem 2.6.27, p. 158]{lernerB}. Hence \cite[Chapter 7, Proposition 2, p. 113]{waelb} implies that the inversion on this set (equipped with the topology induced by $S(1,g;\mathcal{L}_b(\tilde{V}))$) is continuous which implies that $\tilde{\mathbf{f}}$ is continuous. However, we will give a direct proof of the continuity of $\tilde{\mathbf{f}}$ in the case $E$ is a locally compact Hausdorff topological space since much of the ideas (and notations) we employ in this case will become useful for the part concerning the assertion that $\tilde{\mathbf{f}}$ is of class $\mathcal{C}^N$ when $E$ is a smooth manifold.\\
\indent Assume first that $g$ is symplectic; thus $g=g^{\#}=g^{\sigma}$. Let $E$ be a locally compact Hausdorff topological space and let $\mathbf{r}:E\rightarrow S(1,g;\mathcal{L}_b(\tilde{V}))$ be a continuous mapping such that $\|\mathbf{r}(\lambda)^w\|_{\mathcal{L}_b(L^2(V;\tilde{V}))}<1$, $\forall \lambda\in E$. Then $(\mathrm{Id}-\mathbf{r}(\lambda)^w)^{-1}= \mathrm{Id}+\sum_{m=1}^{\infty}(\mathbf{r}(\lambda)^w)^m$ as operators on $L^2(V;\tilde{V})$. Fix $\lambda_0\in E$ and a compact neighbourhood $K$ of $\lambda_0$. There exists $0<\varepsilon<1$ such that $\sup_{\lambda\in K}\|\mathbf{r}(\lambda)^w\|_{\mathcal{L}_b(L^2(V;\tilde{V}))}\leq \varepsilon$ and, for every $k\in\NN$ there exists $\tilde{C}_k\geq 1$ such that $\sup_{\lambda\in K} \|\mathbf{r}(\lambda)\|_{S(1,g;\mathcal{L}_b(\tilde{V}))}^{(k)}\leq \tilde{C}_k$. Now, in analogous way as in the first part of the proof of \cite[Theorem 2.6.27, p. 158]{lernerB} one deduces that for each $k\in\NN$ there exists $0<\varepsilon_k<1$ and $\tilde{C}'_k\geq 1$ such that $\sup_{\lambda\in K}\|\mathbf{r}(\lambda)^{\# m}\|_{S(1,g;\mathcal{L}_b(\tilde{V}))}^{(k)}\leq \tilde{C}'_k\varepsilon_k^m$. Thus $I+\sum_{m=1}^{\infty}\mathbf{r}(\lambda)^{\# m}$ converges to a continuous function $\mathbf{R}:E\rightarrow S(1,g;\mathcal{L}_b(\tilde{V}))$ such that $\mathbf{R}(\lambda)^w$ is the inverse of $\mathrm{Id}-\mathbf{r}(\lambda)^w$ in $L^2(V;\tilde{V})$; a direct inspection also yields $(I-\mathbf{r}(\lambda))\# \mathbf{R}(\lambda)= \mathbf{R}(\lambda)\#(I-\mathbf{r}(\lambda))=I$, for all $\lambda\in E$. If $E$ is a smooth $p$-dimensional manifold and $\mathbf{r}:E\rightarrow S(1,g;\mathcal{L}_b(\tilde{V}))$ is of class $\mathcal{C}^N$, $1\leq N\leq \infty$, such that $\|\mathbf{r}(\lambda)^w\|_{\mathcal{L}_b(L^2(V;\tilde{V}))}<1$, $\forall \lambda\in E$, we claim that $\lambda\mapsto I+\sum_{m=1}^{\infty}\mathbf{r}(\lambda)^{\# m}$, $E\rightarrow S(1,g;\mathcal{L}_b(\tilde{V}))$, is also of class $\mathcal{C}^N$ (it is continuous by the above considerations). Let $\lambda_0\in E$ be arbitrary but fixed. Let $K$ be a compact neighbourhood of $\lambda_0$ included in a coordinate neighbourhood $U$ with local coordinates $(\lambda^{1},\ldots,\lambda^{p})$ about $\lambda_0$. Let $q\in\ZZ_+$, $q\leq N$, be arbitrary but fixed and denote $\tilde{C}=1+\sup_{\alpha\in\NN^p,\,|\alpha|\leq q}\sup_{\lambda\in K} \|\partial^{\alpha}_{\lambda}\mathbf{r}(\lambda)^w\|_{\mathcal{L}_b(L^2(V;\tilde{V}))} <\infty$. For each $\alpha\in\NN^p$, $|\alpha|\leq q$, $\partial^{\alpha}_{\lambda}(\mathbf{r}(\lambda)^{\# m})$ is a sum of $m^{|\alpha|}$ terms
\beqs
b_{\lambda,1}\#\ldots\# b_{\lambda,m},\,\,\, \mbox{with}\,\,\, b_{\lambda,j}=\partial^{\beta^{(j)}}\mathbf{r}(\lambda),\, j=1,\ldots,m,\,\,\, \mbox{and}\,\,\, \sum_{j=1}^m\beta^{(j)}=\alpha.
\eeqs
For $k\in\NN$, we apply \cite[Theorem 2.6.12, p. 145]{lernerB} with $\tau\in[0,1)$ to be chosen later to deduce the existence of $C'_{k,\tau}\geq 1$ and $l_{k,\tau}\in\ZZ_+$ such that
\beqs
\|(b_{\lambda,1}\#\ldots\# b_{\lambda,m})^w\|^{(k)}_{op(1,g)}\leq C'_{k,\tau}\left(\|(b_{\lambda,1}\#\ldots\# b_{\lambda,m})^w\|^{(0)}_{op(1,g)}\right)^{\tau} \left(\|b_{\lambda,1}^w\ldots b_{\lambda,m}^w\|^{(l_{k,\tau})}_{op(1,g)}\right)^{1-\tau}.
\eeqs
When $m\geq 2q$, at least $m-q\geq m/2$ of the terms $b_{\lambda,j}$ are just $\mathbf{r}(\lambda)$ and thus
\beqs
\|(b_{\lambda,1}\#\ldots\# b_{\lambda,m})^w\|^{(0)}_{op(1,g)} \leq C'\|b_{\lambda,1}^w\ldots b_{\lambda,m}^w\|_{\mathcal{L}_b(L^2(V;\tilde{V}))}\leq C' \tilde{C}^q\|\mathbf{r}(\lambda)^w\|_{\mathcal{L}_b(L^2(V;\tilde{V}))}^{m/2}.
\eeqs
To estimate $\|b_{\lambda,1}^w\ldots b_{\lambda,m}^w\|^{(l_{k,\tau})}_{op(1,g)}$, we apply Lemma \ref{estfortheprodlem} to conclude
\begin{multline*}
\|b_{\lambda,1}^w\ldots b_{\lambda,m}^w\|^{(l_{k,\tau})}_{op(1,g)}\leq C_{l_{k,\tau}} (m+1)^{N_{l_{k,\tau}}}\left(C_0\max_{j=1,\ldots,m} \|b_{\lambda,j}\|^{(k_0)}_{S(1,g;\mathcal{L}_b(\tilde{V}))}\right)^{m-4l_{k,\tau}}\\
\cdot \left(C_1\max_{|\beta|\leq q} \|\partial^{\beta}_{\lambda}\mathbf{r}(\lambda)\|^{(k_1(l_{k,\tau}))}_{S(1,g;\mathcal{L}_b(\tilde{V}))}\right)^{4l_{k,\tau}}.
\end{multline*}
Denote $\tilde{C}_1=\max_{|\beta|\leq q}\sup_{\lambda\in K} \|\partial^{\beta}_{\lambda}\mathbf{r}(\lambda)\|^{(k_0)}_{S(1,g;\mathcal{L}_b(\tilde{V}))}<\infty$ (recall, $k_0$ depends only on the structure constants of $g$). Then, for $m\geq\max\{2q,4l_{k,\tau}\}$ we deduce
\begin{multline*}
\|(\partial^{\alpha}_{\lambda}\mathbf{r}(\lambda)^{\# m})^w\|^{(k)}_{op(1,g)}\leq C''_{k,\tau}\tilde{C}^qm^q(m+1)^{N_{l_{k,\tau}}} \|\mathbf{r}(\lambda)^w\|_{\mathcal{L}_b(L^2(V;\tilde{V}))}^{\tau m/2}(C_0\tilde{C}_1)^{(m-4l_{k,\tau})(1-\tau)}\\
\cdot\left(C_1\sup_{\substack{\lambda\in K\\ |\beta|\leq q}} \|\partial^{\beta}_{\lambda}\mathbf{r}(\lambda)\|^{(k_1(l_{k,\tau}))}_{S(1,g;\mathcal{L}_b(\tilde{V}))}\right)^{4l_{k,\tau}(1-\tau)}.
\end{multline*}
We take $\tau\in(0,1)$ such that (recall, $C_0$ depends only on the structure constants of $g$)
\beqs
\varepsilon=\sup_{\lambda\in K}\|\mathbf{r}(\lambda)^w\|_{\mathcal{L}_b(L^2(V;\tilde{V}))}^{\tau/2}(C_0\tilde{C}_1)^{1-\tau}<1
\eeqs
and deduce that $\|(\partial^{\alpha}_{\lambda}\mathbf{r}(\lambda)^{\# m})^w\|^{(k)}_{op(1,g)}\leq C'''m^q(m+1)^{N_{l_k,\tau}}\varepsilon^m$, for $m\geq \max\{2q,4l_{k,\tau}\}$. As $q$ and $\lambda_0$ are arbitrary, we conclude that $\lambda\mapsto 1+\sum_{m=1}^{\infty}\mathbf{r}(\lambda)^{\# m}$, $E\mapsto S(1,g;\mathcal{L}_b(\tilde{V}))$, is of class $\mathcal{C}^N$.\\
\indent Let $\mathbf{f}$ be as in the statement of the theorem; we continue to assume $g$ is symplectic. Fix $\lambda_0\in E$ and let $K$ be a compact neighbourhood of $\lambda_0$ and, if $E$ is a smooth manifold, assume further that $K$ is a regular compact set (i.e. $K=\overline{\mathrm{int}\, K}$) included in a coordinate neighbourhood $U$ of $\lambda_0$. Since $(\mathbf{f}(\lambda)^w)^*\mathbf{f}(\lambda)^w$ is positive invertible on $L^2(V;\tilde{V})$, $\mathbf{f}$ is continuous and $K$ compact, there exists $C>0$ and for each $\lambda\in K$ there exists $0<c_{\lambda}\leq C$ such that
\beqs
c_{\lambda}\|u\|^2_{L^2(V;\tilde{V})}\leq ((\mathbf{f}(\lambda)^w)^*\mathbf{f}(\lambda)^wu,u) \leq C\|u\|^2_{L^2(V;\tilde{V})},\,\,\, \forall u\in L^2(V;\tilde{V}),\, \forall\lambda\in K.
\eeqs
Define $\mathbf{r}(\lambda)=I-C^{-1}\mathbf{f}(\lambda)^*\#\mathbf{f}(\lambda)$ and thus, $\mathbf{r}(\lambda)^w=\mathrm{Id}-C^{-1}(\mathbf{f}(\lambda)^w)^*\mathbf{f}(\lambda)^w$. The mapping $\mathbf{r}:K\rightarrow S(1,g;\mathcal{L}_b(\tilde{V}))$ is continuous and $\|\mathbf{r}(\lambda)^w\|_{\mathcal{L}_b(L^2(V;\tilde{V}))}<1$, $\forall \lambda\in K$. If $E$ is a smooth manifold and $\mathbf{f}$ is of class $\mathcal{C}^N$, then $\mathbf{r}:\mathrm{int}\, K\rightarrow S(1,g;\mathcal{L}_b(\tilde{V}))$ is also of class $\mathcal{C}^N$. As $K$ is compact, we infer $\sup_{\lambda\in K}\|\mathbf{r}(\lambda)^w\|_{\mathcal{L}_b(L^2(V;\tilde{V}))}<1$. The first part now implies that there exists a continuous mapping $\mathbf{R}_K:K\rightarrow S(1,g;\mathcal{L}_b(\tilde{V}))$ such that
\beq\label{htsnvc135}
\mathbf{R}_K(\lambda)\#\mathbf{f}(\lambda)^*\#\mathbf{f}(\lambda)=I=\mathbf{f}(\lambda)^*\#\mathbf{f}(\lambda)\# \mathbf{R}_K(\lambda),\,\,\, \forall \lambda\in K.
\eeq
Similarly, there exists a continuous mapping $\tilde{\mathbf{R}}_K:K\rightarrow S(1,g;\mathcal{L}_b(\tilde{V}))$ such that
\beq\label{vltnke157}
\tilde{\mathbf{R}}_K(\lambda)\#\mathbf{f}(\lambda)\#\mathbf{f}(\lambda)^*=I=\mathbf{f}(\lambda)\# \mathbf{f}(\lambda)^*\# \tilde{\mathbf{R}}_K(\lambda),\,\,\, \forall \lambda\in K.
\eeq
Now, (\ref{htsnvc135}) and (\ref{vltnke157}) imply that $\mathbf{R}_K(\lambda)\#\mathbf{f}(\lambda)^* =\mathbf{f}(\lambda)^*\# \tilde{\mathbf{R}}_K(\lambda)$, $\forall \lambda\in K$. Thus, by defining $\tilde{\mathbf{f}}_K(\lambda)=\mathbf{R}_K(\lambda)\#\mathbf{f}(\lambda)^*$, we deduce that $\tilde{\mathbf{f}}_K:K\rightarrow S(1,g;\mathcal{L}_b(\tilde{V}))$ is continuous and satisfies the conclusion of the theorem on $K$. If $E$ is smooth manifold and $\mathbf{f}$ is of class $\mathcal{C}^N$, the first part verifies that the restriction of $\mathbf{R}_K$ to $\mathrm{int}\, K$ is of class $\mathcal{C}^N$ and thus the restriction of $\tilde{\mathbf{f}}_{K}$ to $\mathrm{int}\, K$ is also of class $\mathcal{C}^N$. Covering $E$ by compact neighbourhoods and noticing that, when $K\cap K'\neq\emptyset$,
\beqs
\tilde{\mathbf{f}}_K(\lambda)= \tilde{\mathbf{f}}_K(\lambda)\#\mathbf{f}(\lambda)\# \tilde{\mathbf{f}}_{K'}(\lambda)=\tilde{\mathbf{f}}_{K'}(\lambda),\,\,\, \forall \lambda\in K\cap K',
\eeqs
we conclude the proof of the theorem when $g$ is symplectic.\\
\indent Assume now that $g$ is a general geodesically temperate H\"ormander metric; $g^{\#}$ is also a H\"ormander metric by \cite[Proposition 2.2.20, p. 78]{lernerB}, $g^{\#}$ is geodesically temperate (cf. \cite[Remark 2.6.21, p. 153]{lernerB}) and every admissible weight for $g$ is admissible for $g^{\#}$ too. Let $E$ be a locally compact Hausdorff topological space. Since $S(1,g;\mathcal{L}_b(\tilde{V}))$ is continuously included into $S(1,g^{\#};\mathcal{L}_b(\tilde{V}))$, the first part of the proof verifies the existence of a continuous map $\tilde{\mathbf{f}}:E\rightarrow S(1,g^{\#};\mathcal{L}_b(\tilde{V}))$ such that (\ref{invconmap}) holds. In fact, \cite[Theorem 2.6.27, p. 158]{lernerB} verifies that $\tilde{\mathbf{f}}(E)\subseteq S(1,g;\mathcal{L}_b(\tilde{V}))$; we only need to prove it is continuous as an $S(1,g;\mathcal{L}_b(\tilde{V}))$-valued mapping. Let $k\in\ZZ_+$. Given $S\in W$ and $T_j\in W$, $j=1,\ldots,k$, satisfying $g_{S}(T_j)=1$, define the function $M_S^{T_{l_1},\ldots,T_{l_m}}(X)=\prod_{j=1}^m g_X(T_{l_j})^{1/2}$, $X\in W$. One easily verifies that $M_S^{T_{l_1},\ldots,T_{l_m}}$, $\{l_1,\ldots,l_m\}\subseteq \{1,\ldots,k\}$, are admissible weights for $g$ and $g^{\#}$ with uniform structure constants for $g$ and $g^{\#}$ (cf. the proof of \cite[Theorem 2.6.27, p. 158]{lernerB}); of course the structure constants depend on $k$. Notice that $\partial_{T_{l_1}}\ldots\partial_{T_{l_m}}\mathbf{f}(\lambda)\in S(M_S^{T_{l_1},\ldots,T_{l_m}},g^{\#};\mathcal{L}_b(\tilde{V}))$, for all $\lambda\in E$. Moreover,
\beq\label{vhnklt159}
\|\partial_{T_{l_1}}\ldots\partial_{T_{l_m}}\mathbf{f}(\lambda)\|^{(q)}_ {S(M_S^{T_{l_1},\ldots,T_{l_m}},g^{\#};\mathcal{L}_b(\tilde{V}))} \leq \|\mathbf{f}(\lambda)\|^{(q+m)}_{S(1,g;\mathcal{L}_b(\tilde{V}))},\,\,\, \mbox{for all}\,\, q\in\NN,\, \lambda\in E.
\eeq
Applying $\partial_{T_1}$ to the identity (\ref{invconmap}), we infer $\partial_{T_1}\tilde{\mathbf{f}}(\lambda)\#\mathbf{f}(\lambda) +\tilde{\mathbf{f}}(\lambda)\#\partial_{T_1}\mathbf{f}(\lambda)=0$ and thus $\partial_{T_1}\tilde{\mathbf{f}}(\lambda)= -\tilde{\mathbf{f}}(\lambda)\#\partial_{T_1}\mathbf{f}(\lambda) \#\tilde{\mathbf{f}}(\lambda)$. By induction, one can verify that $\partial_{T_1}\ldots\partial_{T_k}\tilde{\mathbf{f}}(\lambda)$ is a finite sum of terms of the form
\beq\label{vshrnk157}
\pm f^{(1)}_{\lambda}\#\ldots \#f^{(s)}_{\lambda}
\eeq
where each $f^{(j)}_{\lambda}$ is either $\tilde{\mathbf{f}}(\lambda)$ or $\partial_{T_{l_1}}\ldots\partial_{T_{l_m}}\mathbf{f}(\lambda)$ and $s\leq 2k+1$; furthermore each $\partial_{T_j}$, $j=1\ldots,k$, appears exactly once in (\ref{vshrnk157}). Fix $\lambda_0\in E$ and a compact neighbourhood $K$ of $\lambda_0$. Then $\partial_{T_1}\ldots \partial_{T_k}\tilde{\mathbf{f}}(\lambda)-\partial_{T_1}\ldots \partial_{T_k}\tilde{\mathbf{f}}(\lambda_0)$ is a finite sum of terms of the form
\beq\label{vholrn135}
\pm \left(f^{(1)}_{\lambda}\#\ldots \#f^{(s)}_{\lambda}-f^{(1)}_{\lambda_0}\#\ldots \#f^{(s)}_{\lambda_0}\right)
\eeq
with $f^{(j)}_{\lambda}$ and $f^{(j)}_{\lambda_0}$ as above and $s\leq 2k+1$. The quantity (\ref{vholrn135}) is equal to
\beqs
\pm\sum_{j=1}^sf^{(1)}_{\lambda_0}\#\ldots\# f^{(j-1)}_{\lambda_0}\#(f^{(j)}_{\lambda}-f^{(j)}_{\lambda_0}) \#f^{(j+1)}_{\lambda}\#\ldots\#f^{(s)}_{\lambda}.
\eeqs
We take the norm $\|\cdot\|^{(0)}_{S(M_S^{T_1,\ldots,T_k},g^{\#};\mathcal{L}_b(\tilde{V}))}$ of the above sum. Because of \cite[Theorem 2.3.7, p. 91]{lernerB}, there exists $p\in\ZZ_+$ and $C>0$ independent of $S$ and $T_j$ (since $M_S^{T_{l_1},\ldots, T_{l_m}}$ have uniform structure constants with respect to $g^{\#}$ and $s\leq 2k+1$) such that this norm is dominated by
\beq\label{kncftl157}
C\sum_{j=1}^s \left(\prod_{l=1}^{j-1}\|f^{(l)}_{\lambda_0}\|^{(p)}_{S(\tilde{M}_l,g^{\#};\mathcal{L}_b(\tilde{V}))}\right) \|f^{(j)}_{\lambda}-f^{(j)}_{\lambda_0}\|^{(p)}_{S(\tilde{M}_j,g^{\#};\mathcal{L}_b(\tilde{V}))} \left(\prod_{l=j+1}^s\|f^{(l)}_{\lambda}\|^{(p)}_{S(\tilde{M}_l,g^{\#};\mathcal{L}_b(\tilde{V}))}\right),
\eeq
where $\tilde{M}_j$, $j=1,\ldots,s$, are given as follows: when $f^{(j)}_{\lambda}=\tilde{\mathbf{f}}(\lambda)$ then $\tilde{M}_j(X)=1$, $\forall X\in W$, and when $f^{(j)}_{\lambda}=\partial_{T_{l_1}}\ldots\partial_{T_{l_m}}\mathbf{f}(\lambda)$ then $\tilde{M}_j(X)= M_S^{T_{l_1},\ldots,T_{l_m}}(X)$, $\forall X\in W$. Since $\tilde{\mathbf{f}}$ and $\mathbf{f}$ are continuous with values in $S(1,g^{\#};\mathcal{L}_b(\tilde{V}))$ and $S(1,g;\mathcal{L}_b(\tilde{V}))$ respectively and $K$ is compact, (\ref{kncftl157}) tends to $0$ as $\lambda\rightarrow \lambda_0$ uniformly in $S\in W$, $T_j\in W$, $j=1,\ldots,k$, satisfying $g_S(T_j)=1$ (cf. (\ref{vhnklt159})). Thus
\beqs
\sup_{\substack{S\in W\\ T_1,\ldots,T_k\in W,\, g_S(T_j)=1}}\|\partial_{T_1}\ldots \partial_{T_k}\tilde{\mathbf{f}}(\lambda)(S)-\partial_{T_1}\ldots \partial_{T_k}\tilde{\mathbf{f}}(\lambda_0)(S)\|_{\mathcal{L}_b(\tilde{V})}\rightarrow 0,\,\,\, \mbox{as}\,\, \lambda\rightarrow \lambda_0.
\eeqs
We conclude $\tilde{\mathbf{f}}$ is continuous at $\lambda_0$ as a $S(1,g;\mathcal{L}_b(\tilde{V}))$-valued mapping.\\
\indent Assume that $E$ is a smooth $p$-dimensional manifold with $\mathbf{f}$ being of class $\mathcal{C}^N$, $1\leq N\leq\infty$. Then the first part proves that $\tilde{\mathbf{f}}:E\rightarrow S(1,g^{\#};\mathcal{L}_b(\tilde{V}))$ is of class $\mathcal{C}^N$ and the above also verifies that $\tilde{\mathbf{f}}:E\rightarrow S(1,g;\mathcal{L}_b(\tilde{V}))$ is continuous. For a (local or global) smooth vector field $X$ on $E$, (\ref{invconmap}) implies
\beq\label{vkrtlk159}
X\tilde{\mathbf{f}}(\lambda)=-\tilde{\mathbf{f}}(\lambda)\# X\mathbf{f}(\lambda)\#\tilde{\mathbf{f}}(\lambda),\,\,\, \mbox{as}\,\, S(1,g^{\#};\mathcal{L}_b(\tilde{V}))-\mbox{valued mappings}.
\eeq
To prove $\tilde{\mathbf{f}}$ is of class $\mathcal{C}^N$ as an $S(1,g;\mathcal{L}_b(\tilde{V}))$-valued mapping, let $\lambda_0\in E$ and let $K$ be a regular compact set containing $\lambda_0$ in its interior and $K$ is contained in a coordinate neighbourhood $U$ of $\lambda_0$ with local coordinates $(\lambda^1,\ldots,\lambda^p)$. We infer $\partial_{\lambda^j}\tilde{\mathbf{f}}(\lambda)= -\tilde{\mathbf{f}}(\lambda)\#\partial_{\lambda^j}\mathbf{f}(\lambda)\#\tilde{\mathbf{f}}(\lambda)$ as $S(1,g^{\#};\mathcal{L}_b(\tilde{V}))$-valued maps (cf. (\ref{vkrtlk159})). Hence, the maps $\tilde{\mathbf{f}}_j:U\rightarrow S(1,g;\mathcal{L}_b(\tilde{V}))$, $\tilde{\mathbf{f}}_j(\lambda)= -\tilde{\mathbf{f}}(\lambda)\#\partial_{\lambda^j}\mathbf{f}(\lambda)\#\tilde{\mathbf{f}}(\lambda)$, $j=1,\ldots,p$, are well defined and continuous. We will prove that
\beq\label{nrkvct159}
|\lambda-\lambda_0|^{-1}\left(\tilde{\mathbf{f}}(\lambda)-\tilde{\mathbf{f}}(\lambda_0)- \sum_{j=1}^p(\lambda^j-\lambda^j_0)\tilde{\mathbf{f}}_j(\lambda_0)\right)\rightarrow 0,\,\,\, \mbox{as}\,\, \lambda\rightarrow \lambda_0,\,\,\, \mbox{in}\,\, S(1,g;\mathcal{L}_b(\tilde{V})).
\eeq
Let $k\in\ZZ_+$ be arbitrary but fixed and let $T_1,\ldots,T_k,S\in W$ be such that $g_S(T_j)=1$, $j=1,\ldots,k$. We keep the same notations as above. Notice that $\partial_{T_{l_1}}\ldots\partial_{T_{l_m}}\partial_{\lambda^j}\mathbf{f}(\lambda)\in S(M_S^{T_{l_1},\ldots,T_{l_m}},g^{\#};\mathcal{L}_b(\tilde{V}))$, for all $\lambda\in E$, $j=1,\ldots,p$, and
\beq\label{vncrte157}
\|\partial_{T_{l_1}}\ldots\partial_{T_{l_m}}\partial_{\lambda^j}\mathbf{f}(\lambda)\|^{(q)}_ {S(M_S^{T_{l_1},\ldots,T_{l_m}},g^{\#};\mathcal{L}_b(\tilde{V}))} \leq \|\partial_{\lambda^j}\mathbf{f}(\lambda)\|^{(q+m)}_{S(1,g;\mathcal{L}_b(\tilde{V}))},
\eeq
for all $q\in\NN$, $\lambda\in U$, $j=1,\ldots,p$. Because of (\ref{vshrnk157}) and the fact that $\partial_{T_j}$ commute with $\partial_{\lambda^q}$, we deduce that $\partial_{T_1}\ldots\partial_{T_k}\tilde{\mathbf{f}}_q(\lambda)$ is a finite sum of terms of the form
\beqs
\pm\partial_{\lambda^q}(f^{(1)}_{\lambda}\#\ldots \#f^{(s)}_{\lambda})=\pm\sum_{j=1}^sf^{(1)}_{\lambda}\#\ldots\#\partial_{\lambda^q}f^{(j)}_{\lambda}\#\ldots \#f^{(s)}_{\lambda},
\eeqs
and $s\leq 2k+1$. Notice that $\partial_{T_1}\ldots\partial_{T_k}\tilde{\mathbf{f}}(\lambda)-\partial_{T_1}\ldots\partial_{T_k}\tilde{\mathbf{f}}(\lambda_0)$ is a finite sum of terms of the form (\ref{vholrn135}) which in turn is equal to
\begin{align}
\pm&\sum_{j=1}^sf^{(1)}_{\lambda_0}\#\ldots\# f^{(j-1)}_{\lambda_0}\#\left(f^{(j)}_{\lambda}-f^{(j)}_{\lambda_0}-\sum_{q=1}^p (\lambda^q-\lambda^q_0) \partial_{\lambda^q}f^{(j)}_{\lambda_0}\right)\# f^{(j+1)}_{\lambda}\#\ldots\#f^{(s)}_{\lambda}\label{vstnkc135}\\
\pm& \sum_{q=1}^p (\lambda^q-\lambda^q_0)\sum_{j=1}^s\sum_{t=j+1}^sf^{(1)}_{\lambda_0}\#\ldots\# f^{(j-1)}_{\lambda_0}\# \partial_{\lambda^q}f^{(j)}_{\lambda_0}\# f^{(j+1)}_{\lambda_0}\#\ldots\# f^{(t-1)}_{\lambda_0}\nonumber\\
&\quad\quad\# (f^{(t)}_{\lambda}-f^{(t)}_{\lambda_0})\#f^{(t+1)}_{\lambda}\#\ldots\#f^{(s)}_{\lambda}\label{vsklpr157}\\
\pm& \sum_{q=1}^p (\lambda^q-\lambda^q_0)\partial_{\lambda^q}(f^{(1)}_{\lambda}\#\ldots \#f^{(s)}_{\lambda})\big|_{\lambda=\lambda_0}.\nonumber
\end{align}
We deduce that the derivatives $\partial_{T_1}\ldots\partial_{T_k}$ of the term in brackets in (\ref{nrkvct159}) is a finite sum of terms of the form (\ref{vstnkc135}) and (\ref{vsklpr157}). We take the norm $\|\cdot\|^{(0)}_{S(M_S^{T_1,\ldots,T_k},g^{\#};\mathcal{L}_b(\tilde{V}))}$ of (\ref{vstnkc135}) and (\ref{vsklpr157}) and, similarly as above, we conclude that it is dominated by
\begin{align*}
&\sum_{j=1}^s\left(\prod_{l=1}^{j-1}\|f^{(l)}_{\lambda_0}\|^{(p)}_{S(\tilde{M}_l,g^{\#};\mathcal{L}_b(\tilde{V}))}\right) \left\|f^{(j)}_{\lambda}-f^{(j)}_{\lambda_0}-\sum_{q=1}^p (\lambda^q-\lambda^q_0) \partial_{\lambda^q}f^{(j)}_{\lambda_0}\right\|^{(p)}_{S(\tilde{M}_j,g^{\#};\mathcal{L}_b(\tilde{V}))}\\
&\quad\quad\cdot \left(\prod_{l=j+1}^s\|f^{(l)}_{\lambda}\|^{(p)}_{S(\tilde{M}_l,g^{\#};\mathcal{L}_b(\tilde{V}))}\right)\\
& + |\lambda-\lambda_0|\sum_{q=1}^p\sum_{j=1}^s\sum_{t=j+1}^s \left(\prod_{l=1}^{j-1}\|f^{(l)}_{\lambda_0}\|^{(p)}_{S(\tilde{M}_l,g^{\#};\mathcal{L}_b(\tilde{V}))}\right) \|\partial_{\lambda^q}f^{(j)}_{\lambda_0}\|^{(p)}_{S(\tilde{M}_j,g^{\#};\mathcal{L}_b(\tilde{V}))}\\
&\quad\quad\cdot \left(\prod_{l=j+1}^{t-1}\|f^{(l)}_{\lambda_0}\|^{(p)}_{S(\tilde{M}_l,g^{\#};\mathcal{L}_b(\tilde{V}))}\right) \|f^{(t)}_{\lambda}-f^{(t)}_{\lambda_0}\|^{(p)}_{S(\tilde{M}_t,g^{\#};\mathcal{L}_b(\tilde{V}))} \left(\prod_{l=t+1}^s\|f^{(l)}_{\lambda}\|^{(p)}_{S(\tilde{M}_l,g^{\#};\mathcal{L}_b(\tilde{V}))}\right).
\end{align*}
Thus, the seminorm $\|\cdot\|^{(k)}_{S(1,g;\mathcal{L}_b(\tilde{V}))}$ of (\ref{nrkvct159}) tends to $0$ as $\lambda\rightarrow \lambda_0$ (cf. (\ref{vhnklt159}) and (\ref{vncrte157}); additionally each $s$ is at most $2k+1$). Since $k$ and $\lambda_0\in \mathrm{int}\, K$ are arbitrary, we conclude that $\tilde{\mathbf{f}}$ is $\mathcal{C}^1$ as $S(1,g;\mathcal{L}_b(\tilde{V}))$-valued mapping whose partial derivatives are $\tilde{\mathbf{f}}_j$ (recall, these are continuous $S(1,g;\mathcal{L}_b(\tilde{V}))$-valued mappings). In the same way one proves that $\tilde{\mathbf{f}}$ is $\mathcal{C}^k$, for every $k\in\ZZ_+$, $k\leq N$, i.e. it is of class $\mathcal{C}^N$ on $\mathrm{int}\, K$, and, as $K$ is arbitrary, it is of class $\mathcal{C}^N$ on $E$ as $S(1,g;\mathcal{L}_b(\tilde{V}))$-valued mapping.
\end{proof}


\section{Fredholmness and ellipticity}\label{vnklst135}

The goal of this section is to investigate the relationship between the property of a pseudodifferential operator to restrict to a Fredholm operator between appropriate Sobolev spaces and the notion of ellipticity. In fact, we will prove that these are the same provided the metric is geodesically temperate and its associate function $\lambda_g$ tends to infinity at infinity.\\
\indent We start with the following simple but useful result.

\begin{lemma}\label{compaoper}
Let $g$ be a H\"ormander metric and $M_1$, $M_2$ and $M$ $g$-admissible weights. If $MM_2/M_1$ vanishes at infinity then for any $a\in S(M,g;\mathcal{L}_b(\tilde{V}))$, $a^w$ restricts to a compact operator from $H(M_1,g;\tilde{V})$ into $H(M_2,g;\tilde{V})$.
\end{lemma}

\begin{proof} By \cite[Corollary 2.6.16, p. 150]{lernerB}, we can choose $a_j\in S(M_j,g)$, $\tilde{a}_j\in S(1/M_j,g)$ satisfying $a_j\#\tilde{a}_j=1=\tilde{a}_j\#a_j$, $j=1,2$. Then $a^w=(\tilde{a}_2I)^w((a_2I)\#a\#(\tilde{a}_1I))^w(a_1I)^w$. Since $(a_2I)\#a\#(\tilde{a}_1I)\in S(MM_2/M_1,g;\mathcal{L}_b(\tilde{V}))$ and $MM_2/M_1$ vanishes at infinity, \cite[Theorem 5.5]{hormander} yields that $((a_2I)\#a\#(\tilde{a}_1I))^w$ is compact on $L^2(V;\tilde{V})$ and the result of the lemma follows.
\end{proof}

The definition of ellipticity is as follows.

\begin{definition}
Let $g$ be a H\"ormander metric and $M$ $g$-admissible weight. We say that $a\in S(M,g;\mathcal{L}_b(\tilde{V}))$ is $S(M,g;\mathcal{L}_b(\tilde{V}))$-elliptic if there exist a compact neighbourhood of the origin $K\subseteq W$ and $C>0$ such that $|\det a(X)|\geq CM(X)^{\dim \tilde{V}}$, for all $X\in W\backslash K$.
\end{definition}

\begin{remark}
Of course, in the scalar valued case, this definition reduces to the familiar one when working in the frequently used calculi (the Shubin calculus, the SG calculus, etc.; cf. \cite{NR,Shubin}); see also \cite{BN} for the notion of hypoellipticity in the scalar-valued setting of the Weyl-H\"ormander calculus.
\end{remark}

\begin{remark}
For $a\in S(M,g;\mathcal{L}_b(\tilde{V}))$, we always have $\det a\in S(M^{\dim\tilde{V}},g)$. Thus, for a given $a\in S(M,g;\mathcal{L}_b(\tilde{V}))$, the $S(M,g;\mathcal{L}_b(\tilde{V}))$-ellipticity of $a$ is equivalent to the $S(M^{\dim\tilde{V}},g)$-ellipticity of $\det a$.
\end{remark}

\begin{remark}\label{vcnlpb135}
There exists $c'_0\geq 1$ which depends only on $\dim\tilde{V}$ and $\|\cdot\|_{\tilde{V}}$ such that for any invertible $A:\tilde{V}\rightarrow\tilde{V}$ we have $1/\|A\|_{\mathcal{L}_b(\tilde{V})}\leq \|A^{-1}\|_{\mathcal{L}_b(\tilde{V})}\leq c'_0\|A\|^{\dim\tilde{V}-1}_{\mathcal{L}_b(\tilde{V})}/|\det A|$. Consequently, for $a\in S(M,g;\mathcal{L}_b(\tilde{V}))$ the $S(M,g;\mathcal{L}_b(\tilde{V}))$-ellipticity of $a$ is equivalent to the following: there exist a compact neighbourhood of the origin $K\subseteq W$ and $C>0$ such that $a(X)$ is invertible on $W\backslash K$ and $\|a(X)^{-1}\|_{\mathcal{L}_b(\tilde{V})}\leq C/M(X)$, $\forall X\in W\backslash K$.
\end{remark}

\begin{theorem}\label{tlckpe157}
Let $g$ be a H\"ormander metric satisfying $\lambda_g\rightarrow \infty$ and $M$ a $g$-admissible weight. If $a\in S(M,g;\mathcal{L}_b(\tilde{V}))$ is elliptic than for any $g$-admissible weight $M_1$, $a^w$ restricts to a Fredholm operator from $H(M_1,g;\tilde{V})$ into $H(M_1/M,g;\tilde{V})$ and its index is independent of $M_1$.
\end{theorem}

\begin{proof} Let $\tilde{a}=a^{-1}$ away from the origin and modified on a sufficiently large compact neighbourhood of the origin so as to be a well defined element of $S(1/M,g;\mathcal{L}_b(\tilde{V}))$. Then $\tilde{a}\#a-I\in S(1/\lambda_g,g;\mathcal{L}_b(\tilde{V}))$ and Lemma \ref{compaoper} verifies that $\tilde{a}^wa^w-\mathrm{Id}$ is compact operator on $H(M_1,g;\tilde{V})$. Similarly, $a^w\tilde{a}^w-\mathrm{Id}$ is compact operator on $H(M_1/M,g;\tilde{V})$. Consequently, $a^w:H(M_1,g;\tilde{V})\rightarrow H(M_1/M,g;\tilde{V})$ is Fredholm. To prove that the index is independent of $M_1$, let $M_2$ be another $g$-admissible weight and denote by $A_j$ the restriction of $a^w$ to $H(M_j,g;\tilde{V})\rightarrow H(M_j/M,g;\tilde{V})$, $j=1,2$. Because of \cite[Corollary 2.6.16, p. 150]{lernerB} we can choose $b_1\in S(M_1/M_2,g)$ and $b_2\in S(M_2/M_1,g)$ such that $b_1\#b_2=1=b_2\#b_1$. Consequently, the restrictions $B_1$ and $B_2$ of $(b_1I)^w$ and $(b_2I)^w$ to $H(M_1,g;\tilde{V})\rightarrow H(M_2,g;\tilde{V})$ and $H(M_2/M,g;\tilde{V})\rightarrow H(M_1/M,g;\tilde{V})$ respectively, are isomorphisms. Since $(b_2I)\#a\#(b_1I)-a\in S(M/\lambda_g,g;\mathcal{L}_b(\tilde{V}))$ and $\lambda_g\rightarrow \infty$ at infinity, Lemma \ref{compaoper} implies that $B_2A_2B_1-A_1: H(M_1,g;\tilde{V})\rightarrow H(M_1/M,g;\tilde{V})$ is compact. Consequently, $\operatorname{ind} A_2=\operatorname{ind} B_2A_2B_1=\operatorname{ind} A_1$.
\end{proof}

\begin{remark}
If there exists $C,\delta>0$ such that $\lambda_g(X)\geq C(1+g_0(X))^{\delta}$, $\forall X\in W$, (i.e. if the metric satisfies the strong uncertainty principle) then given an elliptic $a\in S(M,g)$ one can construct a parametrix of $a$ (see \cite{schrohe1,NR1}; see also \cite{CancellerChemin,Helfer}) and derive from that the the index of $a^w_{|H(M_1,g)}:H(M_1,g)\rightarrow H(M_1/M,g)$ does not depend on $M_1$; in fact one can derive the stronger result that the dimensions of the kernel and cokernel are independent of $M_1$ (cf. \cite[Section 1.6]{NR}). The significance of the above result is that the index is independent of $M_1$ even when only requiring $\lambda_g\rightarrow \infty$; however we can not say anything about the invariance of the dimensions of the kernel and cokernel.
\end{remark}

Our next goal is to prove a converse result to that of Theorem \ref{tlckpe157}; namely, if $a^w$ restrict to a Fredholm operator between Sobolev spaces than it is elliptic. The proof relies on Theorem \ref{invconsmoot} and, consequently, on the spectral invariance of the Weyl-H\"ormander calculus which, in turn, relies on the geodesic temperance of $g$. We first prove this result for symbols in $S(1,g;\mathcal{L}_b(\tilde{V}))$ and derive the general case from the latter.\\
\indent Before we proceed, we need the the following result whose proof is the same as for \cite[Lemma 2.7]{schrohe} and we omit it.

\begin{lemma}\label{lemmaforprope}
Let $g$ be a H\"ormander metric and $a\in S(1,g;\mathcal{L}_b(\tilde{V}))$ is such that $A=a^w{}_{|L^2(V;\tilde{V})}$ has finite dimensional range. Then there exist $\varphi_j\in\SSS(V;\tilde{V}')$, $\psi_j\in\SSS(V;\tilde{V})$, $j=1,\ldots,m$, such that $Af=\sum_{j=1}^m\langle f,\varphi_j\rangle\psi_j$, $f\in L^2(V;\tilde{V})$. Consequently, the kernel of $A$ is in $\SSS(V;\tilde{V}')\otimes \SSS(V;\tilde{V})$ and thus $a\in \SSS(W;\mathcal{L}_b(\tilde{V}))$.
\end{lemma}

\begin{theorem}\label{fredelliptforoneordsym}
Let $g$ be a geodesically temperate H\"ormander metric satisfying $\lambda_g\rightarrow \infty$. If $a\in S(1,g;\mathcal{L}_b(\tilde{V}))$ is such that $a^w$ restricts to a Fredholm operator on $L^2(V)$ then $a$ is elliptic.
\end{theorem}

\begin{proof} Throughout the proof, we fix an inner product on $\tilde{V}$ and denote by $\|\cdot\|_{\tilde{V}}$ and $\|\cdot\|_{\mathcal{L}_b(\tilde{V})}$ the induced norms. Denote $A=a^w{}_{|L^2(V;\tilde{V})}$. As $A$ is Fredholm, $0$ is an isolated point of the spectrum of the positive operator $A^*A$ (see \cite[Lemma 7.2]{cordes}). Let $\Gamma$ be a circle about the origin in $\CC$ with radius $r\leq 1$ which contains no other point of the spectrum of $A^*A$ except possibly $0$ and define
\beqs
B=\frac{1}{2\pi i}\int_{\Gamma} (\lambda\mathrm{Id}-A^*A)^{-1}d\lambda.
\eeqs
Then $B$ is an orthogonal projection and \cite[Section 5.10, Theorems 10.2 and 10.1, p. 330]{tal} imply that the range of $B$ is $\operatorname{ker} A^*A=\operatorname{ker} A$; i.e. $B$ is an orthogonal projection onto $\operatorname{ker} A$ (this trivially holds if $\operatorname{ker} A=\{0\}$). Let $\tilde{a}_{\lambda}=\lambda I-a^*\#a\in S(1,g;\mathcal{L}_b(\tilde{V}))$, $\lambda\in \Gamma$. The mapping $\lambda\mapsto \tilde{a}_{\lambda}$, $\Gamma\rightarrow S(1,g;\mathcal{L}_b(\tilde{V}))$, is continuous (and in fact smooth) and $\tilde{a}_{\lambda}^w$ is invertible on $L^2(V;\tilde{V})$. Theorem \ref{invconsmoot} yields the existence of continuous (and in fact smooth) mapping $\lambda\mapsto \tilde{b}_{\lambda}$, $\Gamma\rightarrow S(1,g;\mathcal{L}_b(\tilde{V}))$, such that $\tilde{b}_{\lambda}\# \tilde{a}_{\lambda}=I=\tilde{a}_{\lambda}\# \tilde{b}_{\lambda}$, $\lambda\in \Gamma$. Define
\beqs
b(X)=\frac{1}{2\pi i}\int_{\Gamma}\tilde{b}_{\lambda}(X)d\lambda= \frac{r}{2\pi}\int_0^{2\pi} \tilde{b}_{re^{it}}(X) e^{it}dt,\,\,\,\, X\in W.
\eeqs
Clearly $b\in\mathcal{C}^{\infty}(W;\mathcal{L}_b(\tilde{V}))$ and, since $\lambda\mapsto\tilde{b}_{\lambda}$ is continuous and $\Gamma$ is compact, one easily derives that $b\in S(1,g;\mathcal{L}_b(\tilde{V}))$. For each $m\in\ZZ_+$, define $\tilde{c}_{m,t}=\tilde{b}_{re^{2\pi i j/m}}e^{2\pi i j/m}$, when $2\pi(j-1)/m\leq t< 2\pi j/m$, $j=1,\ldots, m$; clearly $c_{m,t}\in S(1,g;\mathcal{L}_b(\tilde{V}))$. Furthermore,
\beqs
b_m=\frac{r}{2\pi}\int_0^{2\pi} c_{m,t}dt=\frac{r}{2\pi}\sum_{j=1}^m\tilde{b}_{re^{2\pi i j/m}}e^{2\pi i j/m}\cdot \frac{2\pi}{m}\in S(1,g;\mathcal{L}_b(\tilde{V})).
\eeqs
Now
\beqs
\|b^w-b_m^w\|_{\mathcal{L}_b(L^2(V;\tilde{V}))}&\leq& C\|b-b_m\|^{(k)}_{S(1,g;\mathcal{L}_b(\tilde{V}))}\\
&\leq& \frac{C}{2\pi}\sum_{j=1}^m\int_{2\pi (j-1)/m}^{2\pi j/m}\|\tilde{b}_{re^{it}}e^{it}-\tilde{b}_{re^{2\pi i j/m}}e^{2\pi i j/m}\|^{(k)}_{S(1,g;\mathcal{L}_b(\tilde{V}))}dt.
\eeqs
The right hand side tends to $0$ as $m\rightarrow \infty$ since $t\mapsto \tilde{b}_{re^{it}}e^{it}$, $[0,2\pi]\rightarrow S(1,g;\mathcal{L}_b(\tilde{V}))$, is uniformly continuous. Consequently $b_m^w\rightarrow b^w$ in $\mathcal{L}_b(L^2(V;\tilde{V}))$. On the other hand, $c_{m,t}^w\rightarrow \tilde{b}_{re^{it}}^we^{it}$, as $m\rightarrow \infty$, pointwise in $\mathcal{L}_b(L^2(V;\tilde{V}))$, so dominated convergence implies $b_m^w\rightarrow B$ in $\mathcal{L}_b(L^2(V;\tilde{V}))$. We conclude $b^w{}_{|L^2(V;\tilde{V})}=B$. Since the range of $B$ is the finite dimensional space $\mathrm{ker}\, A$, we can apply Lemma \ref{lemmaforprope} to deduce $b\in \SSS(W;\mathcal{L}_b(\tilde{V}))$. One easily verifies that $B+A^*A$ is invertible on $L^2(V;\tilde{V})$ and consequently, there exists $c\in S(1,g;\mathcal{L}_b(\tilde{V}))$ such that $c^w{}_{|L^2(V;\tilde{V})}= (B+A^*A)^{-1}$. We infer $c\#(b+a^*\#a)=I$ which yields $c\#a^*\#a=I-c\#b$. Since $c\#b\in \SSS(W;\mathcal{L}_b(\tilde{V}))$ and $c\#a^*\#a-ca^*a\in S(1/\lambda_g,g;\mathcal{L}_b(\tilde{V}))$, we deduce $ca^*a-I\in S(1/\lambda_g,g;\mathcal{L}_b(\tilde{V}))$. As $c\in S(1,g;\mathcal{L}_b(\tilde{V}))$ and $1/\lambda_g$ vanishes at infinity, we conclude the validity of the theorem.\footnote{there exists $\varepsilon>0$ which depends only on $\dim \tilde{V}$ and $\|\cdot\|_{\tilde{V}}$ such that for all $P\in\mathcal{L}(\tilde{V})$ satisfying $\|P\|_{\mathcal{L}_b(\tilde{V})}\leq \varepsilon$ it holds $|\det(I+P)|\geq 1/2$}
\end{proof}

The main result of the section is the following.

\begin{theorem}
Let $g$ be a geodesically temperate H\"ormander metric satisfying $\lambda_g\rightarrow \infty$ and $M$ and $M_1$ two $g$-admissible weights. If $a\in S(M,g;\mathcal{L}_b(\tilde{V}))$ is such that $a^w$ restricts to a Fredholm operator from $H(M_1,g;\tilde{V})$ into $H(M_1/M,g;\tilde{V})$ then $a$ is elliptic.
\end{theorem}

\begin{proof} Take elliptic $b\in S(1/M_1,g)$ and elliptic $c\in S(M_1/M,g)$. Then $\tilde{a}=(cI)\#a\#(bI)\in S(1,g;\mathcal{L}_b(\tilde{V}))$ and $\tilde{a}^w=(cI)^wa^w(bI)^w$ is Fredholm operator on $L^2(V;\tilde{V})$ (cf. Theorem \ref{tlckpe157}). By Theorem \ref{fredelliptforoneordsym}, $|\det \tilde{a}(X)|\geq 1/C$ and $\|\tilde{a}(X)^{-1}\|_{\mathcal{L}_b(\tilde{V})}\leq C$ for all $X$ outside of a compact neighbourhood of the origin $K\subseteq W$ (cf. Remark \ref{vcnlpb135}). Denote $f=\tilde{a}-(cI)a(bI)\in S(1/\lambda_g,g;\mathcal{L}_b(\tilde{V}))$ and notice that
\beqs
|\det (cI)a(bI)(X)|=|\det \tilde{a}(X)||\det(I-\tilde{a}(X)^{-1}f(X))|,\,\, \forall X\in W\backslash K.
\eeqs
As $1/\lambda_g$ vanishes at infinity the claim in the theorem follows.
\end{proof}

\end{document}